\documentclass[11pt,reqno]{amsart}

\usepackage{mathrsfs}

\usepackage{hyperref}

\usepackage{graphicx}

\usepackage{xcolor}

\usepackage{enumerate}
\usepackage{bm}

\newtheorem{theorem}{Theorem}[section]
\newtheorem{lemma}[theorem]{Lemma}
\newtheorem{proposition}[theorem]{Proposition}
\newtheorem{corollary}[theorem]{Corollary}

\theoremstyle{definition}

\theoremstyle{remark}

\numberwithin{equation}{section}

\title[The boundary case for complex Monge-Amp\`ere type equations]{The boundary case for complex Monge-Amp\`ere type equations}

\author{Wei Sun}

\address{Institute of Mathematical Sciences, ShanghaiTech University, Shanghai, China}
\email{sunwei@shanghaitech.edu.cn}

\begin{document}
\setlength{\baselineskip}{1.2\baselineskip}

\begin{abstract}
In this paper, we shall study the boundary case for complex Monge-Amp\`ere type equations under certain geometric assumptions. 

\end{abstract}


\maketitle


\section {Introduction}

Let $(M,\omega)$ be a compact connected K\"ahler manifold without boundary of complex dimension $n \geq 2$. In this paper, we are concerned with the following complex Monge-Amp\`ere type equations: for a K\"ahler metric $\chi$, 
\begin{equation}
\label{equation-1}
	(\chi + \sqrt{-1} \partial\bar\partial\varphi)^n = c (\chi + \sqrt{-1} \partial\bar\partial \varphi)^m \wedge \omega^{n - m},\quad \sup_M \varphi = 0
\end{equation} 
where $ 0 \leq m < n$ and $ c: = \frac{\int_M \chi^n}{\int_M \chi^m \wedge \omega^{n - m}}$. A function $\varphi$ is called admissible or $\chi$-plurisubharmonic if $\chi + \sqrt{-1} \partial\bar\partial \varphi \geq 0$. 
We plan to find out an admissible solution to Equation~\eqref{equation-1}. In this case, Equation~\eqref{equation-1} is elliptic.

The complex Monge-Amp\`ere type equations are of great importance in complex geometry and analysis. There have been a lot of works on the following problem
\begin{equation}
\label{equation-general}
	(\chi + \sqrt{-1} \partial\bar\partial\varphi)^n = g (\chi + \sqrt{-1} \partial\bar\partial \varphi)^m \wedge \omega^{n - m},\quad \sup_M \varphi = 0 ,
\end{equation} 
with cone condition~\cite{FangLaiMa}\cite{SongWeinkove} (or $\mathcal{C}$-subsolution~\cite{Szekelyhidi})
\begin{equation}
\label{cone-condition}
	n \chi^{n - 1} > m g \chi^{m - 1} \wedge \omega^{n - m} .
\end{equation}
When $m = 0$, Equation~\eqref{equation-general} is exactly complex Monge-Amp\`ere equation, which plays key roles in complex geometry since the foundational work of Yau~\cite{Yau1978}. The cone condition~\eqref{cone-condition} is naturally satisfied by K\"ahler metric $\chi$. When $m = n - 1$, Equation~\eqref{equation-general} is defined in the works of Donaldson~\cite{Donaldson1999} and Chen~\cite{Chen2000} in different problems.

When $g = c$, Equation~\eqref{equation-1} was first studied via parabolic flows by Chen~\cite{Chen2004}, Weinkove~\cite{Weinkove2004}\cite{Weinkove2006}, Song-Weikove~\cite{SongWeinkove}, and Fang-Lai-Ma~\cite{FangLaiMa}. Later, several different elliptic methods were figured out respectively by Li-Shi-Yao~\cite{LiShiYao}, Sun~\cite{Sun2016}\cite{Sun2017}, Collins-Sz\'ekelyhidi~\cite{CollinsSzekelyhidi} and Sz\'ekelyhidi~\cite{Szekelyhidi}.

An interesting question is what will happen when cone condtion~\eqref{cone-condition} degenerates to the boundary case
\begin{equation}
\label{condition-1}
	n \chi^{n - 1} \geq m c \chi^{m - 1} \wedge \omega^{n - m} .
\end{equation}
Fang, Lai,Song and Weinkove~\cite{FangLaiSongWeinkove} studied 
\begin{equation*}
	(\chi + \sqrt{-1}\partial\bar\partial \varphi)^2 = c (\chi + \sqrt{-1} \partial\bar\partial\varphi) \wedge \omega
\end{equation*} 
on K\"ahler surfaces in the boundary case 	$2 \chi \geq c \omega $.
Instead of Equation~\eqref{equation-1}, we shall consider in this paper
\begin{equation}
\label{equation-2}
	(\chi + \tilde\chi + \sqrt{-1} \partial\bar\partial\varphi)^n = c (\chi + \tilde\chi + \sqrt{-1}\partial\bar\partial \varphi)^m \wedge \omega^{n - m},\quad \sup_M \varphi = 0,
\end{equation}
where $\tilde \chi$ is semipositive and big,
\begin{equation}
\label{equality-1-3}
	\int_M (\chi + \tilde\chi  )^n = c\int_M (\chi + \tilde\chi  )^m \wedge \omega^{n - m} ,
\end{equation}
and the boundary case~\eqref{condition-1} holds. 

The following theorem is our main result.
\begin{theorem}
\label{main-theorem}
Let $(M, \omega)$ be a compact K\"ahler manifold without boundary of complex dimension $n \geq  2$ , $\chi$ also a K\"ahler metric,  and $\tilde \chi$ a big semipositive form. Suppose that $\chi$ satisfies the boundary case~\eqref{condition-1} of cone condition. Then there is a bounded weak solution $\varphi$ in pluripotential sense solving Equation~\eqref{equation-2}. In particular, $\varphi$ is smooth in the ample locus of $\tilde \chi$.

\end{theorem}

To construct the weak solution, we shall solve the following approximation equation: for $1\geq t > 0$
\begin{equation}
\label{approximation-equation-1-5}
\begin{aligned}
	\left(\chi +  t \chi +  \tilde \chi  + \sqrt{-1} \partial\bar\partial\varphi_t\right)^n 
	= c   \left(\chi + t\chi+  \tilde \chi + \sqrt{-1} \partial\bar\partial \varphi_t\right)^m \wedge \omega^{n - m} + b_t f \omega^n , 
\end{aligned}
\end{equation}
where $\sup_M \varphi_t = 0$ and $f$ is smooth and positive. Moreover, it is satisfied that 
\begin{equation}
\label{condition-1-6}
	 \int_M f \omega^n = \int_M \omega^n, \qquad \int_M (\chi + t\chi + \tilde \chi )^n = c\int_M (\chi + t\chi +  \tilde \chi )^m \wedge \omega^{n - m} + b_t \int_M \omega^n .
\end{equation}
A key tool in the argument is a new PDE proof of $L^\infty$ estimates~\cite{SuiSun2021}, which is derived from the work of Guo-Phong-Tong~\cite{GPT2021}. 
Guo, Phong and Tong combined the methods of Wang-Wang-Zhou~\cite{WangWangZhou} and Chen-Cheng~\cite{ChenCheng} to prove the $L^\infty$ estimate for Hessian type equations. Recently, Sui and the author~\cite{SuiSun2021} adapted the estimation argument to Hessian quotient equations.

An essential integrant  in our argument is the sempositve big form $\tilde \chi$. 
It was observed in \cite{FangLaiSongWeinkove} that on K\"ahler surfaces, the boundary case~\eqref{condition-1} implies the existence of $\tilde \chi$ via a simple transformation to complex Monge-Amp\`ere equation, which does not work well for higher dimensions.  However, we have a direct corollary.
\begin{corollary}
Let $(M, \omega)$ be compact K\"ahler manifold without boundary of complex dimension $n \geq  2$ , $\chi$ also a K\"ahler metric. 
Suppose that class $\left[\chi - \left(\frac{m c}{n}\right)^{\frac{1}{n - m}} \omega\right]$ is big and has a semipositve representative form.
Then there is a bounded weak solution $\varphi$ in pluripotential sense solving Equation~\eqref{equation-1}. In particular, $\varphi$ is smooth in the ample locus of $\tilde \chi$.

\end{corollary}

An immediate question is  that under what kind of geometric conditions, we can find out the semipositive form $\tilde \chi$ on K\"ahler manifolds of dimension $n > 2$. 
There have been some great works recently on numerical characterizations of K\"ahler class~\cite{Demailly-Paun}\cite{Chen2021}. 
In view of these works, we wish to find out numerical conditions in future.

It is worth a mention that we can weaken the smoothness of $\tilde \chi$ to $C^{2,\alpha}$,
and only require that there is a $C^2$ admissible function $\underline u$ satisfying 
\begin{equation*}
\label{boundary-case-weak}
n (\chi + \sqrt{-1} \partial\bar\partial \underline u)^{n - 1} \geq m g (\chi + \sqrt{-1} \partial\bar\partial \underline u)^{m - 1} \wedge \omega^{n - m} .
\end{equation*}

 This paper is organized as follows.
 In Section~\ref{L-estimate}, we shall obtain $t$-independent $L^\infty$ estimate for approximation  Equation~\eqref{approximation-equation-1-5}.  
 In Section~\ref{solvability}, we shall prove that  the approximation equation has a smooth solution for $t > 0$, which may blow up as $t$ approaches $0$.   
 In Section~\ref{smoothness}, we shall further prove that the smoothness of solution in the ample locus of $\tilde\chi$ is independent of $t$. 
 In Section~\ref{stability}, a stability estimate will be provided.  
 In Section~\ref{solution}, we shall construct a weak solution in pluripotential sense according to the estimates in previous sections. Moreover, the solution is also well defined in other senses.  
 In Section~\ref{uniqueness}, we shall prove that the uniqueness of weak solution. 
 In Section~\ref{fake}, we shall explain the reason why we only discuss the special case of constant right-side term in Theorem~\ref{main-theorem}.


\medskip
\section{$L^\infty$ estimate}
\label{L-estimate}

In this section, we shall prove the $L^\infty$ estimate for approximation equation~\eqref{approximation-equation-1-5}, which is independent of parameter $t$. In this paper, the constant $C$ may vary line by line and independent of parameter $t$, unless otherwise indicated.

In the argument, De Giorgi iteration is an important tool. For a proof, we refer the readers to the books of Gilbarg-Trudinger~\cite{GilbargTrudinger} and Chen-Wu~\cite{ChenWu}.
 \begin{lemma}[De Giorgi iteration]
 \label{de-giorgi-iteration}
 Suppose that $\phi (s)$ is a nonnegative increasing function on $[s_0,+\infty)$, and satisfies 
 \begin{equation}
 s'^\alpha	\phi (s' + s) \leq C \phi^\beta (s), \qquad \forall s' > 0 , s \geq s_0,
 \end{equation}
 where $\alpha > 0$ and $\beta > 1$. Then we have
 \begin{equation}
 \phi (s_0 + d) = 0,
 \end{equation}
 where
 \begin{equation}
 	d = C^{\frac{1}{\alpha}} \phi^{\frac{\beta -1}{\alpha}} (s_0) 2^{\frac{\beta}{\beta - 1}} .
 \end{equation}

 \end{lemma}

 In approximation equation~\eqref{approximation-equation-1-5}, it must be that $b_t > 0$ for any $t> 0$, as
 \begin{equation*}
 \begin{aligned}
 	& \frac{d}{dt} \left(\int_M (\chi + t\chi + \tilde \chi )^n - c\int_M (\chi + t\chi +  \tilde \chi )^m \wedge \omega^{n - m}\right) \\
 	=& \int_M \chi \wedge \left( n (\chi + t\chi + \tilde \chi)^{n - 1} - mc (\chi + t\chi + \tilde \chi)^{m -  1} \wedge \omega^{n - m} \right) \\
 	> &\; 0 .
 \end{aligned}
 \end{equation*}
 Then we have the following $L^\infty$ estimate independent of parameter $t$ when $0 < t \leq 1$. The argument follows that in \cite{SuiSun2021}. A proof is provided here in order to check the uniform dependence. 
\begin{theorem}
\label{theorem-3-2}
Suppose that $\Vert f \Vert_{L^q} < + \infty$  for $q > 1$. 
Then there is a constant $C$ independent of  parameter $t$ so that for all $t \in (0,1]$ and admissible solution $\varphi_t$ to Equation~\eqref{approximation-equation-1-5}, 
it holds that $- \varphi_t < C$.
\end{theorem}

\begin{proof}

If $\hat \chi \geq 0$ and 
\begin{equation*}
(\chi + \hat\chi)^n = c (\chi + \hat\chi)^m \wedge\omega^{n - m} + b_t f \omega^n ,
\end{equation*}
then we have
\begin{equation*}
\begin{aligned}
	\hat\chi^n
	\leq& \sum^n_{i = 1} C^i_n \hat \chi^i \wedge \chi^{n - i} - c \sum^m_{i = 1} C^i_m \hat\chi^{ i} \wedge \chi^{m - i} \wedge \omega^{n-m} \\
	=& (\chi + \hat \chi)^n - c (\chi + \hat\chi)^m \wedge \omega^{n - m} - (\chi^n - c \chi^m\wedge \omega^{n - m})
	\\
	\leq& (b_1 f + C_0) \omega^n
	.
\end{aligned}
\end{equation*}
In this paper, $\hat\chi$ is an auxiliary notation, which will be used to express different terms in later sections.

Let
$$
V_t := \int_M (t\chi + \tilde \chi)^n \geq \int_M \tilde \chi^n > 0 .
$$
We solve the auxiliary complex Monge-Amp\`ere equation
\begin{equation*}
	(t\chi + \tilde \chi + \sqrt{-1} \partial\bar\partial \psi_{t,k})^n = \dfrac{\tau_k (- \varphi_t - s)}{A_{s,k}} (b_1 f + C_0) \omega^n ,
\end{equation*}
where
\begin{equation*}
	A_{s,k} =  \frac{1}{V_t} \int_M \tau_k (-\varphi_t - s) (b_1 f + C_0) \omega^n .
\end{equation*}
We have
\begin{equation*}
	\lim_{k \to \infty} A_{s,k} = A_s := \frac{1}{V_t} \int_{M_s} (- \varphi_t - s) (b_1 f + C_0) \omega^n ,
\end{equation*}
where
\begin{equation*}
	M_s := \{\varphi_t \leq - s\} .
\end{equation*}

We shall study the function 
\begin{equation*}
	\Phi := - \left(\frac{n + 1}{n}\right)^{\frac{n}{n + 1}} A^{\frac{1}{n + 1}}_{s,k}  \left( - \psi_{t,k} + \frac{n}{n + 1} A_{s,k} \right)^{\frac{n}{n + 1}} - \varphi_t - s .
\end{equation*}
By the argument of maximum principle~\cite{GPT2021}\cite{GPTW2021}\cite{SuiSun2021}, we obtain that  
$\Phi \leq 0$ on $M$.
We can find a constant $\alpha_0 > 0$ and a constant $C > 0$ so that
\begin{equation*}
\int_M \exp \left(\frac{(n + 1) \alpha_0}{n}  \psi\right) \omega^n  \leq C
\end{equation*}
for any smooth function $\psi$ with $2 \chi + \tilde \chi + \sqrt{-1}\partial\bar\partial \psi \geq 0$.
Therefore
\begin{equation}
\label{inequality-3-8}
	\int_{M}  \exp \left( \frac{(n + 1) \alpha_0}{n} \psi_{t,k}\right)\omega^n \leq C,
\end{equation}
and
\begin{equation}
\label{bound-3-9}
\begin{aligned}	
	A_s 
	&\leq \frac{1}{V_t} \int_{M} - \varphi_t  (b_1 f + C_0) \omega^n \\
	&\leq \frac{ \Vert b_1 f + C_0\Vert_{L^q} }{V_t} \left(\int_M |\varphi_t|^{\frac{q}{q - 1}} \omega^n \right)^{1 - \frac{1}{q}}\\
	&\leq \frac{C (\alpha_0, q, \Vert   f  \Vert_{L^q} )}{V_t} \left(\int_M  \exp \left( - \frac{(n + 1) \alpha_0}{n} \varphi_t\right)\omega^n \right)^{1 - \frac{1}{q}} \\
	&\leq \frac{C (\alpha_0, q, \Vert   f  \Vert_{L^q} )}{\int_M \tilde \chi^n} 
	.
\end{aligned}
\end{equation}
From \eqref{inequality-3-8},
\begin{equation}
\label{inequality-3-9}
	\begin{aligned}
		&\quad \int_{M_s} \exp \left(  \frac{\alpha_0}{A^{\frac{1}{n}}_{s,k}}  ( -\varphi_t -  s)^{\frac{n + 1}{n}}\right) \omega^n \\
		&\leq  \int_{M_s} \exp \left( \frac{(n + 1) \alpha_0}{n} \psi_{t,k} + \alpha_0 A_{s,k}\right) \omega^n \\
		&\leq C \exp \left(\alpha_0 A_{s,k}\right) .
	\end{aligned}
\end{equation}
Letting $k \to \infty$ in \eqref{inequality-3-9}, we have
\begin{equation}
\label{inequality-3-10}
 \int_{M_s} \exp \left(  \frac{\alpha_0}{A^{\frac{1}{n}}_{s}}  ( -\varphi_t -  s)^{\frac{n + 1}{n}}\right) \omega^n \leq C \exp \left(\alpha_0 A_{s} \right) .
\end{equation}
For $q > 1$ and $p > n$, 
\begin{equation}
\label{inequality-3-11}
\begin{aligned}
	&\quad 
	\frac{\alpha^p_0}{A^{\frac{p}{n}}_s} \int_{M_s} \left(- \varphi_t - s\right)^{\frac{(n + 1) p}{n}}  (b_1 f + C_0) \omega^n \\
	&\leq           \left(\int_{M_s}      \left( \frac{\alpha_0}{A^{\frac{1}{n}}_s} \left(- \varphi_t - s\right)^{\frac{n + 1}{n}} \right)^\frac{pq}{q - 1} \omega^n      \right)^{1 - \frac{1}{q}}                   \Vert b_1 f + C_0\Vert_{L^q} \\
	&\leq C (p,q)    \left(\int_{M_s}    \exp  \left( \frac{\alpha_0}{A^{\frac{1}{n}}_s} \left(- \varphi_t - s\right)^{\frac{n + 1}{n}} \right) \omega^n      \right)^{1 - \frac{1}{q}}                   \Vert b_1 f + C_0\Vert_{L^q} \\
	&\leq C(p,q, \Vert f \Vert_{L^q}) \exp \left(\alpha_0 \left(1 - \frac{1}{q}\right) E\right)       
	,
\end{aligned}
\end{equation}
where $E$ is an uniform upper bound for $A_s$ as in \eqref{bound-3-9}.
The first inequality in \eqref{inequality-3-11} is from H\"older inequality with respect to measure $\omega^n$, while the third is from \eqref{inequality-3-10}. 
For $p > n$, applying H\"older inequality with respect to measure $  (b_1 f + C_0) \omega^n$,
\begin{equation}
\label{inequality-3-12}
\begin{aligned}
	A_s &= \frac{1}{V_t} \int_{M_s} (-\varphi_t - s) (b_1 f + C_0) \omega^n \\
	&\leq 
	\frac{1}{V_t} \left( \int_{M_s} (-\varphi_t - s)^{\frac{(n + 1) p}{n}} (b_1 f + C_0) \omega^n \right)^{\frac{n}{(n + 1) p}}
	 \left(  \int_{M_s}  (b_1 f + C_0) \omega^n \right)^{\frac{(n + 1) p - n}{(n + 1) p}} \\
	&\leq 
	\frac{ C(p,q, \Vert f \Vert_{L^q}) }{V_t} \left( \frac{A^{\frac{p}{n}}_s}{a^p_0}  \exp \left(\alpha_0 \left(1 - \frac{1}{q}\right) E\right)     \right)^{\frac{n}{(n + 1)p}}  	\left(  \int_{M_s}  (b_1 f + C_0) \omega^n \right)^{\frac{(n + 1) p - n}{(n + 1) p}} \\
	&\leq \frac{C (p,q,\alpha_0, E, \Vert f \Vert_{L^q})}{V_t} A^{\frac{1}{n + 1}}_s  	\left( \int_{M_s}  (b_1 f + C_0) \omega^n \right)^{\frac{(n + 1) p - n}{(n + 1) p}}
	,
\end{aligned}
\end{equation}
and hence
\begin{equation*}
	A_s \leq C (p,q,\alpha_0, E, \Vert f \Vert_{L^q}) \frac{1}{V^{\frac{n + 1}{ n}}_t} 	\left(  \int_{M_s}  (b_1 f + C_0) \omega^n \right)^{\frac{(n + 1) p - n}{n  p}} .
\end{equation*}
The second inequality in \eqref{inequality-3-12} is from \eqref{inequality-3-11}, and we can choose $p = n + 1$ in this proof.
For any $s' > 0$ and $s \geq 0$,
\begin{equation}
\label{inequality-3-14}
\begin{aligned}
	s' \int_{M_{s + s'}} (b_1 f + C_0) \omega^n
	&\leq \int_{M_{s + s'}} (- \varphi_t - s) (b_1 f + C_0) \omega^n \\
	&\leq \int_{M_{s }} (- \varphi_t - s) (b_1 f + C_0) \omega^n \\
	&\leq C (p,q,\alpha_0, E, \Vert f \Vert_{L^q}) \frac{1}{V^{\frac{ 1}{ n}}_t} 	\left(  \int_{M_s}  (b_1 f + C_0) \omega^n \right)^{\frac{(n + 1) p - n}{n  p}} \\
	&\leq \dfrac{C (p,q,\alpha_0, E, \Vert f \Vert_{L^q}) }{  \left( \int_M \tilde \chi^n   \right)^{\frac{1}{n}}}  	\left(  \int_{M_s}  (b_1 f + C_0) \omega^n \right)^{1 + \frac{1}{n} - \frac{1}{p}} 
	.
\end{aligned}
\end{equation}
Combining Lemma~\ref{de-giorgi-iteration} and Inequality~\eqref{inequality-3-14}, it can be concluded that 
\begin{equation*}
	\int_{M_s} (b_1 f + C_0) \omega^n = 0
\end{equation*}
when
\begin{equation*}
	s 
	\geq 
	s_\infty
	:=
	\dfrac{C (p,q,\alpha_0, E, \Vert f \Vert_{L^q}) }{  \left( \int_M \tilde \chi^n   \right)^{\frac{1}{n}}} 2^{\frac{1 + \frac{1}{n} - \frac{1}{p}}{\frac{1}{n} - \frac{1}{p}}} \left( b_1 \left(\int_M \omega^n \right)^{1 - \frac{1}{q}} \Vert f \Vert_{L^q} +  C_0 \int_M \omega^n \right)^{\frac{1}{n} - \frac{1}{p}} .
\end{equation*}

\end{proof}

\medskip

\section{Solvability of Approximation equations}
\label{solvability}

In this section, we shall prove the solvability of approximation equation~\eqref{approximation-equation-1-5} when $t > 0$. The crucial steps are to discover the gradient estimate and partial $C^2$ estimates. The arguments follow those in \cite{Sun2014e}, and also work well on Hermitian manifolds actually. To well understand the dependence of constants and coefficients in Section~\ref{solvability} and \ref{smoothness} with respect to parameter $t$, we include detailed proofs here. In this section, we need global estimates for fixed $t > 0$ in order to prove the existence of smooth admissible solution to approximation equation through standard elliptic methods. In the next section, we shall show the uniform smoothness in the ample locus, which needs $t$-independent estimates. Unfortunately, the  $t$-independent estimates are not global in views of the counterexample in \cite{SongWeinkove}.

We denote $\nabla$ the Chern connection of $\omega$.
As in \cite{GuanLi2010}\cite{GuanLi2012}\cite{Sun2016}, we have
\begin{equation}
	\omega_{i\bar j} = \omega\left(\frac{\partial}{\partial z^i} , \frac{\partial}{\partial\bar z^j}\right) , \qquad \left[ \omega^{i\bar j}\right] = \left[\omega_{i\bar j}\right]^{-1} ,
\end{equation}
in local coordinates $(z^1, \cdots,z^n)$. 
For simplicity, we denote
\begin{equation}
X = \chi + t\chi + \tilde \chi + \sqrt{-1} \partial\bar\partial \varphi_t .
\end{equation}
For real $(1,1)$ forms in this section, we may use the same letter to denote its corresponding matrix in local chart if there is no confusion.  
Thus,
\begin{equation*}
	X_{i\bar j} = X\left(\frac{\partial}{\partial z^i} , \frac{\partial}{\partial\bar z^j}\right) = \chi_{i\bar j} + t\chi_{i\bar j} + \tilde \chi_{i\bar j} + \partial_i\bar\partial_j \varphi_t ,
\end{equation*}
and we denote
$ [X^{i\bar j}] :=	X^{-1} $.
In the reasoning of this paper, we use the indices to represent covariant derivatives with respect to $\nabla$. For more details, we refer the readers to  \cite{GuanLi2010}\cite{GuanLi2012}\cite{Sun2016}.


Let $S_k (\bm{\lambda})$ denote the $k$-th elementary symmetric polynomial of $\bm{\lambda} \in \mathbb{R}^n$,
\begin{equation*}
	S_k (\bm{\lambda}) = \sum_{1 \leq i_1 < \cdots < i_k \leq n} \lambda_{i_1} \cdots \lambda_{i_k} .
\end{equation*}
In particular, $S_0 (\bm{\lambda}) = 1$ by convention. 
Moreover, 
\begin{equation*}
	S_{k;i_1\cdots i_s} (\bm{\lambda}) := S_k (\bm{\lambda}|_{\lambda_{i_1} = \cdots = \lambda_{i_s} = 0}).
\end{equation*}
For a Hermitian matrix $X$, it is well defined that
\begin{equation*}
	S_k (X) := S_k (\bm{\lambda}_* (X))
\end{equation*}
and
\begin{equation*}
	S_k (X^{-1}) := S_k (\bm{\lambda}^* (X)) 
\end{equation*}
where $\bm{\lambda}_* (X)$ and $\bm{\lambda}^* (X)$ denote the eigenvalues of $X$ with respect to $\omega$ and to $\omega^{-1}$, respectively.
Therefore, approximation equation~\eqref{approximation-equation-1-5} can be rewritten in the following two forms,
\begin{equation}
\label{equation-4-2}
S_n (X) = \frac{c}{C^m_n} S_m (X) + b_t f ,
\end{equation}
or
\begin{equation}
\label{equation-4-3}
1 = \frac{c}{C^m_n} S_{n - m} (X^{-1}) + b_t f S_n (X^{-1}) , 
\end{equation}
in local coordinates.

\medskip
\subsection{Gradient estimate}

\begin{theorem}[Gradient estimate]
 Let $\varphi_t \in C^3 (M)$ be an admissible solution to equation~\eqref{approximation-equation-1-5} for $t > 0$. Then there are positive constants $C$ and $\Lambda$ such that 
 \begin{equation*}
 	\sup_M |\nabla \varphi_t|^2 \leq C e^{\Lambda (\varphi_t - \inf_M \varphi_t)} ,
 \end{equation*}
 where $C$ and $\Lambda$ depend on $t$ and given geometric quantities. 
\end{theorem}

\begin{proof}

We shall investigate the following function
\begin{equation}
\label{test-function-4-6}
	\ln |\nabla \varphi_t|^2_\omega - A \left(\varphi_t - \inf_M \varphi_t\right) + \frac{1}{\varphi_t - \inf_M \varphi_t + 1} ,
\end{equation}
where $A >> 1$ is to be specified later. 
In the argument, we shall adapt the approaches of Phong-Sturm~\cite{PhongSturm}\cite{PhongSturm2012} and Blocki~\cite{Blocki2009}.  
The constant $C$ may varies line by line. However it is independent from $A$ if not specified.
To lower the dependence on the properties of function $f$, we shall adopt equation form~\eqref{equation-4-2} to study the gradient estimate.

There is a point $z_0$ where $\ln |\nabla \varphi_t|^2_\omega - A \varphi_t$ reaches its maximal value. 
We pick a local chart around $z_0$ such that $\omega_{i\bar  j} = \delta_{ij}$ and $X_{i\bar j}$ is diagonal at $z_0$.
The first step is to show that at $z_0$, there is a constant $C (A,t) > 0$ such that
\begin{equation}
\label{inequality-4-6}
|\nabla \varphi_t|^2 \leq C(A,t) \left(\varphi_t - \inf_M \varphi_t + 1\right)^6 .
\end{equation}
We have at $z_0$, 
\begin{equation}
\label{inequality-4-7}
	0 = \frac{\partial_l |\nabla \varphi_t|^2}{|\nabla \varphi_t|^2} - A \varphi_{t,l} - \frac{\varphi_{t,l}}{(\varphi_t - \inf_M \varphi_t + 1)^2} ,
\end{equation}
and
\begin{equation}
\label{inequality-4-8}
\begin{aligned}
	0 
	&\geq
	\frac{\bar\partial_l \partial_l |\nabla \varphi_t|^2}{|\nabla \varphi_t|^2} - \frac{\partial_l |\nabla \varphi_t|^2 \bar\partial_l |\nabla \varphi_t|^2}{|\nabla \varphi_t|^4} \\
	&\qquad - \left(A + \frac{1}{(\varphi_t - \inf_M \varphi_t + 1)^2}\right) \varphi_{t,l\bar l}  + \frac{2 \varphi_{t,l} \varphi_{t,\bar l}}{(\varphi_t - \inf_M \varphi_t + 1)^3}  .
\end{aligned}
\end{equation}
Differentiating \eqref{equation-4-2},
\begin{equation}
	b_t f_l = \sum_i \left(S_{n - 1;i} (X) - \frac{c}{C^m_n} S_{m - 1;i} (X) \right) X_{i\bar il} ,
\end{equation}
and
\begin{equation}
\begin{aligned}
	b_t f_{l\bar l}
	&= \sum_i \left(S_{n - 1;i} (X) - \frac{c}{C^m_n} S_{m - 1;i} (X) \right) X_{i\bar il\bar l}  \\
	&\qquad + \sum_{i\neq j} \left(S_{n - 2;ij} (X) - \frac{c}{C^m_n} S_{m - 2;ij} (X) \right) X_{i\bar il} X_{j\bar j\bar l}\\
	&\qquad  - \sum_{i\neq j} \left(S_{n - 2;ij} (X) - \frac{c}{C^m_n} S_{m - 2;ij} (X) \right) X_{i\bar jl} X_{j\bar i\bar l} .
\end{aligned}
\end{equation}
Direct calculation shows that,
\begin{equation}
\label{inequality-4-9}
	\partial_l |\nabla \varphi_t|^2 = \sum_i (\varphi_{t,i} \varphi_{t, l\bar i} + \varphi_{t, il} \varphi_{t,\bar i}) ,
\end{equation}
and
\begin{equation}
\label{inequality-4-10}
\begin{aligned}
	\bar\partial_l\partial_l |\nabla \varphi_t|^2 
	&= 
	\sum_i (\varphi_{t,il} \varphi_{t, \bar i\bar l} + \varphi_{t,i\bar l} \varphi_{t, l\bar i}) + 2 \sum_i \mathfrak{Re} \{X_{l\bar l i} \varphi_{t,\bar i} \} \\
	&\qquad  - 2 \sum_i \mathfrak{Re} \{ (\chi_{l\bar li} + t \chi_{l\bar li} + \tilde \chi_{l\bar li}) \varphi_{t,\bar i}\}  + \sum_{i,j} R_{l\bar l i\bar j} \varphi_{t,j} \varphi_{t,\bar i} .
\end{aligned}
\end{equation}
Substituting \eqref{inequality-4-10} into \eqref{inequality-4-8}, 
\begin{equation}
\label{inequality-4-11}
\begin{aligned}
	0 &\geq
	|\nabla \varphi_t|^2	\sum_i (\varphi_{t,il} \varphi_{t, \bar i\bar l} + \varphi_{t,i\bar l} \varphi_{t, l\bar i}) + 2 |\nabla \varphi_t|^2	\sum_i \mathfrak{Re} \{X_{l\bar l i} \varphi_{t,\bar i} \} \\
	&\quad  - C |\nabla \varphi_t|^3 - C |\nabla \varphi_t|^4  - \partial_l |\nabla \varphi_t|^2 \bar\partial_l |\nabla \varphi_t|^2 \\
	&\quad - |\nabla \varphi_t|^4 \left(A + \frac{1}{(\varphi_t - \inf_M \varphi_t + 1)^2}\right) \varphi_{t,l\bar l}  +  2 |\nabla \varphi_t|^4 \frac{ \varphi_{t,l} \varphi_{t,\bar l}}{(\varphi_t - \inf_M \varphi_t + 1)^3}  
	.
\end{aligned}
\end{equation}
From \eqref{inequality-4-9},
\begin{equation}
\label{equality-4-13}
\begin{aligned}
 	\left|\sum_i \varphi_{t,il} \varphi_{t,\bar i} \right|^2
	&= 
	\left|\partial_l |\nabla \varphi_t|^2 \right|^2 + \left|\sum_i \varphi_{t,i} \varphi_{t,l\bar i} \right|^2  - 2 X_{l\bar l} \mathfrak{Re} \Big\{\partial_l |\nabla \varphi_t|^2   \varphi_{t,\bar i} \Big\} \\
	&\qquad  + 2 \mathfrak{Re} \Big\{ \partial_l |\nabla \varphi_t|^2 \sum_i (\chi_{i\bar l} + t\chi_{i\bar l} + \tilde \chi_{i\bar l})\varphi_{t,\bar i} \Big\} .
\end{aligned}
\end{equation}
Combining \eqref{equality-4-13} and \eqref{inequality-4-7},
\begin{equation}
\label{inequality-4-12}
\begin{aligned}
	&\quad |\nabla \varphi_t|^2 \sum_i \varphi_{t,il} \varphi_{t, \bar i\bar l} -  \left|\partial_l |\nabla \varphi_t|^2\right|^2 \\
	&\geq \left|\sum_i \varphi_{t,il} \varphi_{t, \bar i}\right|^2 -  \left|\partial_l |\nabla \varphi_t|^2\right|^2  \\
	&=
	\left|\sum_i \varphi_{t,i} \varphi_{t,l\bar i}\right|^2 
	- 2 X_{l\bar l}  |\nabla \varphi_t|^2 \left(A + \frac{1}{(\varphi_t - \inf_M \varphi_t + 1)^2}\right) \varphi_{t,l} \varphi_{t,\bar l} \\
		&\qquad  + 2   |\nabla \varphi_t|^2 \left(A + \frac{1}{(\varphi_t - \inf_M \varphi_t + 1)^2}\right) \mathfrak{Re} \Big\{ \varphi_{t,l}\sum_i (\chi_{i\bar l} + t\chi_{i\bar l} + \tilde \chi_{i\bar l})\varphi_{t,\bar i} \Big\} .
\end{aligned}
\end{equation}
Substituting \eqref{inequality-4-12} into \eqref{inequality-4-11} and applying  Cauchy-Schwarz inequality,
\begin{equation}
\label{inequality-4-13}
\begin{aligned}
	0 
		&\geq
	2 |\nabla \varphi_t|^2	\sum_i \mathfrak{Re} \{X_{l\bar l i} \varphi_{t,\bar i} \} - C |\nabla \varphi_t|^3 - C |\nabla \varphi_t|^4   \\
		&\quad - |\nabla \varphi_t|^4 \left(A + \frac{1}{(\varphi_t - \inf_M \varphi_t + 1)^2}\right) \varphi_{t,l\bar l}  +  2 |\nabla \varphi_t|^4 \frac{ \varphi_{t,l} \varphi_{t,\bar l}}{(\varphi_t - \inf_M \varphi_t + 1)^3}  \\
		&\quad - 2 X_{l\bar l}  |\nabla \varphi_t|^2 \left(A + \frac{1}{(\varphi_t - \inf_M \varphi_t + 1)^2}\right) \varphi_{t,l} \varphi_{t,\bar l} \\
				&\quad  - C   |\nabla \varphi_t|^3 \left(A + \frac{1}{(\varphi_t - \inf_M \varphi_t + 1)^2}\right) \\
				&\quad   -   |\nabla \varphi_t|^3 \left(A + \frac{1}{(\varphi_t - \inf_M \varphi_t + 1)^2}\right)   \varphi_{t,l} \varphi_{t,\bar l} 
				.
\end{aligned}
\end{equation}

If $|\nabla \varphi_t| \leq (A + |\nabla f| +  1) \left(\varphi_t - \inf_M \varphi_t + 1\right)^3$, then the first step of the proof is done. 
From now on, we may assume that 
\begin{equation}
\label{assumption-4-14}
	|\nabla \varphi_t| > \left(A + |\nabla f| + 1\right) \left(\varphi_t - \inf_M \varphi_t + 1\right)^3 ,
\end{equation}
and consequently Inequality~\eqref{inequality-4-13} can be reduced to
\begin{equation}
\label{inequality-4-14}
\begin{aligned}
	0 
		&\geq
	2 |\nabla \varphi_t|^2	\sum_i \mathfrak{Re} \{X_{l\bar l i} \varphi_{t,\bar i} \} - C |\nabla \varphi_t|^3 - C |\nabla \varphi_t|^4   \\
		&\quad - |\nabla \varphi_t|^4 \left(A + \frac{1}{(\varphi_t - \inf_M \varphi_t + 1)^2}\right) \varphi_{t,l\bar l}  +   |\nabla \varphi_t|^4 \frac{ \varphi_{t,l} \varphi_{t,\bar l}}{(\varphi_t - \inf_M \varphi_t + 1)^3}  \\
		&\quad - 2 X_{l\bar l}  |\nabla \varphi_t|^2 \left(A + \frac{1}{(\varphi_t - \inf_M \varphi_t + 1)^2}\right) \varphi_{t,l} \varphi_{t,\bar l} \\
				&\quad  - C   |\nabla \varphi_t|^3 \left(A + \frac{1}{(\varphi_t - \inf_M \varphi_t + 1)^2}\right) 
				.
\end{aligned}
\end{equation}
Multiplying \eqref{inequality-4-14} by $S_{n - 1;l} (X) - \frac{c}{C^m_n} S_{m - 1;l} (X)$ and summing them up over index $l$,
\begin{equation} 
\label{inequality-4-16}
\begin{aligned}
	0 
	&\geq
 - |\nabla \varphi_t|^2 \left(A + \frac{1}{(\varphi_t - \inf_M \varphi_t + 1)^2}\right) \sum_l \left(S_{n - 1;l} (X) - \frac{c}{C^m_n} S_{m - 1;l} (X)\right) \varphi_{t,l\bar l} \\
		&\quad  +   \frac{  |\nabla \varphi_t|^2}{(\varphi_t - \inf_M \varphi_t + 1)^3} \sum_l \left(S_{n - 1;l} (X) - \frac{c}{C^m_n} S_{m - 1;l} (X)\right) \varphi_{t,l} \varphi_{t,\bar l}\\
		&\quad - 2     \left(A + 1\right) \sum_l \left(S_{n - 1;l} (X) - \frac{c}{C^m_n} S_{m - 1;l} (X)\right) X_{l\bar l} \varphi_{t,l} \varphi_{t,\bar l} \\
				&\quad  - C   |\nabla \varphi_t| \left(A + 1\right) \sum_l \left(S_{n - 1;l} (X) - \frac{c}{C^m_n} S_{m - 1;l} (X)\right) \\
		&\quad -  C |\nabla \varphi_t|^2  \sum_l \left(S_{n - 1;l} (X) - \frac{c}{C^m_n} S_{m - 1;l} (X)\right)  	   - 2 b_t |\nabla \varphi_t| |\nabla f|   
				.
\end{aligned}
\end{equation}
Since $X$ is positive definite,
\begin{equation}
\begin{aligned}
	\left(S_{n - 1;i} (X) - \frac{c}{C^m_n} S_{m - 1;i} (X) \right) X_{i\bar i}
	&\leq
	\sum_l \left(S_{n - 1;l} (X) - \frac{c}{C^m_n} S_{m - 1;l} (X) \right) X_{l\bar l} \\
	&\leq n S_n (X)
	,
\end{aligned}
\end{equation}
for any $1 \leq i \leq n$. 
Therefore,
\begin{equation}
\label{inequality-4-17}
\begin{aligned}
	0 &\geq
 - |\nabla \varphi_t|^2 \left(A + \frac{1}{(\varphi_t - \inf_M \varphi_t + 1)^2}\right) \sum_l \left(S_{n - 1;l} (X) - \frac{c}{C^m_n} S_{m - 1;l} (X)\right) \varphi_{t,l\bar l} \\
		&\quad  +   \frac{  |\nabla \varphi_t|^2}{(\varphi_t - \inf_M \varphi_t + 1)^3} \sum_l \left(S_{n - 1;l} (X) - \frac{c}{C^m_n} S_{m - 1;l} (X)\right) \varphi_{t,l} \varphi_{t,\bar l} \\
		&\quad - 2   |\nabla \varphi_t|^2 \left(A + 1\right)  \left((n - m) S_n (X) + m b_t f\right) \\
				&\quad  - C   |\nabla \varphi_t| \left(A + 1\right) \sum_l \left(S_{n - 1;l} (X) - \frac{c}{C^m_n} S_{m - 1;l} (X)\right) \\
		&\quad -  C |\nabla \varphi_t|^2  \sum_l \left(S_{n - 1;l} (X) - \frac{c}{C^m_n} S_{m - 1;l} (X)\right)  	   - 2 b_t |\nabla \varphi_t| |\nabla f|   
				.
\end{aligned}
\end{equation}

As shown in \cite{Sun2015p}, $\ln \frac{S_n (\bm{\lambda})}{S_m (\bm{\lambda}) + a} $ is concave with respect to $\bm{\lambda}$, if $a \geq 0$.
Then,
\begin{equation}
\label{inequality-4-18}
\begin{aligned}
	& - \left(A + \frac{1}{(\varphi_t - \inf_M \varphi_t + 1)^2}\right) \sum_l \left(S_{n - 1;l} (X) - \frac{c}{C^m_n} S_{m - 1;l} (X)\right) \varphi_{t,l\bar l} \\
	& +   \frac{  1}{2 (\varphi_t - \inf_M \varphi_t + 1)^3} \sum_l \left(S_{n - 1;l} (X) - \frac{c}{C^m_n} S_{m - 1;l} (X)\right) \varphi_{t,l} \varphi_{t,\bar l} \\
	\geq& 
			\frac{t}{2} \left(A + \frac{1}{(\varphi_t - \inf_M \varphi_t + 1)^2}\right) \sum_l \left(S_{n - 1;l} (X) - \frac{c}{C^m_n} S_{m - 1;l} (X)\right) \chi_{l\bar l} \\
	& + \left(A + \frac{1}{(\varphi_t - \inf_M \varphi_t + 1)^2}\right) S_n (X) \ln \dfrac{S_n (\tilde X)}{S_m (\tilde X) + \frac{b_t C^m_n}{c} f}  
	,
\end{aligned}
\end{equation}
where
\begin{equation*}
\tilde X := \chi + \frac{t}{2} \chi + \tilde \chi + \frac{1}{2 (A + 1) (\varphi_t - \inf_M \varphi_t + 1)^3} \sqrt{-1} \partial \varphi_t \wedge \bar\partial \varphi_t .
\end{equation*}
Now we shall calculate $\frac{S_n (\tilde X)}{S_m (\tilde X) + \frac{b_t C^m_n}{c} f}$. 
Here we denote
\begin{equation*}
	\hat \chi = \frac{t}{4} \chi + \tilde \chi +  \frac{1}{2 (A + 1) (\varphi_t - \inf_M \varphi_t + 1)^3} \sqrt{-1} \partial \varphi_t \wedge \bar\partial \varphi_t .
\end{equation*}
Since $\chi$ satisfies the boundary case of cone condition,
\begin{equation}
\label{inequality-4-19}
\begin{aligned}
	&\quad \tilde X^n -  \left(1 + \frac{t}{4}\right)^{n - m} \left(c \tilde X^m \wedge \omega^{n - m} + b_t f \omega^n \right) \\
	&\geq \left(1 + \frac{t}{4}\right)^n
	\left(   \frac{1}{\left(1 + \frac{t}{4}\right)^n} \hat \chi^n + \chi^n - c \chi^m \wedge \omega^{n - m} - \frac{b_t}{\left(1 + \frac{t}{4}\right)^m} f \omega^n \right) \\
	&\geq 
	 \frac{n}{(A + 1) (\varphi_t - \inf_M \varphi_t + 1)^3} \sqrt{-1}\partial\varphi_t \wedge \bar\partial\varphi_t \wedge \left(\frac{t}{4} \chi + \tilde \chi\right)^{n - 1} \\
	 &\qquad - \left(1 + \frac{t}{4}\right)^n   c\chi^m \wedge \omega^{n - m}   - 	\left(1 + \frac{t}{4}\right)^{n - m} b_t f \omega^n
	.
\end{aligned}
\end{equation}
It is obvious that there is a constant $C > 1$ depending on $t$ such that
\begin{equation}
	\tilde X^n -  \left(1 + \frac{t}{4}\right)^{n - m} \left(c \tilde X^m \wedge \omega^{n - m} + b_t f \omega^n \right) > 0
\end{equation}
whenever
\begin{equation}
\label{inequality-4-21}
	|\nabla \varphi_t|^2 \geq C (A + 1) (\varphi_t - \inf_M \varphi_t + 1)^3 .
\end{equation}
Inequality~\eqref{inequality-4-19} then implies that
\begin{equation}
\label{inequality-4-22}
	\ln \dfrac{S_n (\tilde X)}{S_m (\tilde X) + \frac{b_t C^m_n}{c} f}  >  \ln \left(1 + \frac{t}{4}\right)^{n - m} \geq \frac{4}{5} (n - m) t .
\end{equation}
if \eqref{inequality-4-21} holds true.
Substituting \eqref{inequality-4-22} into \eqref{inequality-4-18},
\begin{equation}
\label{inequality-4-23}
\begin{aligned}
	& - \left(A + \frac{1}{(\varphi_t - \inf_M \varphi_t + 1)^2}\right) \sum_l \left(S_{n - 1;l} (X) - \frac{c}{C^m_n} S_{m - 1;l} (X)\right) \varphi_{t,l\bar l} \\
	& +   \frac{  1}{(\varphi_t - \inf_M \varphi_t + 1)^3} \sum_l \left(S_{n - 1;l} (X) - \frac{c}{C^m_n} S_{m - 1;l} (X)\right) \varphi_{t,l} \varphi_{t,\bar l} \\
	> &\, 
			\frac{A t}{2}  \sum_l \left(S_{n - 1;l} (X) - \frac{c}{C^m_n} S_{m - 1;l} (X)\right) \chi_{l\bar l} +  \frac{4  (n - m) A t }{5}S_n (X)  \\
		& 
		+   \frac{  1}{2 (\varphi_t - \inf_M \varphi_t + 1)^3} \sum_l \left(S_{n - 1;l} (X) - \frac{c}{C^m_n} S_{m - 1;l} (X)\right) \varphi_{t,l} \varphi_{t,\bar l} 
	.
\end{aligned}
\end{equation}
Substituting \eqref{assumption-4-14} and \eqref{inequality-4-23} into \eqref{inequality-4-17},
\begin{equation}
\begin{aligned}
	0 &\geq
 				\frac{A t}{2}   \sum_l \left(S_{n - 1;l} (X) - \frac{c}{C^m_n} S_{m - 1;l} (X)\right) \chi_{l\bar l} +  \frac{4  (n - m) A t }{5}   S_n (X)  \\
 		&\quad +   \frac{  1}{2 (\varphi_t - \inf_M \varphi_t + 1)^3} \sum_l \left(S_{n - 1;l} (X) - \frac{c}{C^m_n} S_{m - 1;l} (X)\right) \varphi_{t,l} \varphi_{t,\bar l}  \\
		&\quad - 2  n  \left(A + 1\right)   S_n (X) -  C  \sum_l \left(S_{n - 1;l} (X) - \frac{c}{C^m_n} S_{m - 1;l} (X)\right) \\
		&\quad -  C   \sum_l \left(S_{n - 1;l} (X) - \frac{c}{C^m_n} S_{m - 1;l} (X)\right)  	   -  2 b_t
				.
\end{aligned}
\end{equation}
There is a constant $\sigma_0 > 0$ such that $\chi > \delta_0 \omega$, and consequently it holds that
\begin{equation}
\begin{aligned}
	0 &\geq
 				\frac{A \delta_0 t}{2}   \sum_l \left(S_{n - 1;l} (X) - \frac{c}{C^m_n} S_{m - 1;l} (X)\right)  +  \frac{4  (n - m) A t }{5}   S_n (X)  \\
 		&\quad +   \frac{  1}{2 (\varphi_t - \inf_M \varphi_t + 1)^3} \sum_l \left(S_{n - 1;l} (X) - \frac{c}{C^m_n} S_{m - 1;l} (X)\right) \varphi_{t,l} \varphi_{t,\bar l} \\
		&\quad - 2  n  \left(A + 1\right)   S_n (X) - 2 C  \sum_l \left(S_{n - 1;l} (X) - \frac{c}{C^m_n} S_{m - 1;l} (X)\right)  -   2 b_t 
				.
\end{aligned}
\end{equation}
If we choose $A$ sufficiently large, which is dependent on $t$,  then
\begin{equation}
\label{inequality-4-27}
\begin{aligned}
	 	 &\quad 2  n  \left(A + 1\right)   S_n (X)    \\
	 	 &\geq 
	 	 \frac{A \delta_0 t}{4}   \sum_l \left(S_{n - 1;l} (X) - \frac{c}{C^m_n} S_{m - 1;l} (X)\right)  +  \frac{  (n - m) A t }{2}   S_n (X) \\
	 	 &\geq \frac{A (n - m)\delta_0 t}{4 n} S_{n - 1}  (X)  +  \frac{  (n - m) A t }{2}   S_n (X)  ,
\end{aligned}
\end{equation}
where the inequalities above are deduced by Newton-Maclaurin inequality. However, Inequality~\eqref{inequality-4-27} implies that
\begin{equation}
\label{inequality-4-28}
	X^{i\bar i} \leq \frac{C}{t} .
\end{equation}
Then we obtain an bound,
\begin{equation}
\begin{aligned}
	&\quad \sum_l \left(S_{n - 1;l} (X) - \frac{c}{C^m_n} S_{m - 1;l} (X)\right) X^2_{l\bar l} \\
	&= S_n (X) \left(b_t f S_{n - 1} (X^{-1})  + (m + 1) \frac{c}{C^m_n}   S_{n - m - 1} (X^{-1})\right) \\
	&\leq C(t) S_n (X) . 
\end{aligned}
\end{equation}
Recalling \eqref{inequality-4-16} and \eqref{inequality-4-23},
\begin{equation} 
\begin{aligned}
	0 
		&\geq 
			\frac{A \delta_0 t}{4}   \sum_l \left(S_{n - 1;l} (X) - \frac{c}{C^m_n} S_{m - 1;l} (X)\right) +  \frac{  (n - m) A t }{2}   S_n (X) \\
 		&\quad +   \frac{  1}{2 (\varphi_t - \inf_M \varphi_t + 1)^3} \sum_l \left(S_{n - 1;l} (X) - \frac{c}{C^m_n} S_{m - 1;l} (X)\right) \varphi_{t,l} \varphi_{t,\bar l} \\
		&\quad - \epsilon C (t) \left(A + 1\right) S_n (X)  - \frac{A + 1}{ \epsilon  |\nabla \varphi_t|^2}  \sum_l \left(S_{n - 1;l} (X) - \frac{c}{C^m_n} S_{m - 1;l} (X)\right)   \varphi_{t,l} \varphi_{t,\bar l} 
				.
\end{aligned}
\end{equation}
While $\epsilon > 0$ is small enough so that
\begin{equation*}
	\frac{(n - m) At}{2} \geq \epsilon C(t) (A + 1) ,
\end{equation*}
we have
\begin{equation}
	\frac{1}{ \epsilon  |\nabla \varphi_t|^2} \left(A + \frac{1}{(\varphi_t - \inf_M \varphi_t + 1)^2}\right)   
		\geq 
			  \frac{  1}{2 (\varphi_t - \inf_M \varphi_t + 1)^3}  
			.
\end{equation}
The first step of gradient estimate is complete.

In the second step, we shall derive the global gradient estimate. Without loss of generality, we may assume
\begin{equation}
\label{inequality-4-33}
	\ln |\nabla \varphi_t|^2 + 1 \geq A \left(\varphi_t - \inf_M \varphi_t\right) .
\end{equation}
Substituting \eqref{inequality-4-33} into \eqref{inequality-4-6},
\begin{equation}
|\nabla \varphi_t|^2 \leq C(A,t) \left(\ln |\nabla \varphi_t|^2 + 1 \right)^6 ,
\end{equation}
which tells us that $|\nabla \varphi_t|^2 < C(A,t)$. Therefore, at any $z\in M$,
\begin{equation*}
\begin{aligned}
	\ln |\nabla \varphi_t (z)|^2 
	&\leq               
	          \ln |\nabla \varphi_t (z_0)|^2    -   A \left(\varphi_t (z_0) - \inf_M \varphi_t\right) + \frac{1}{\varphi_t (z_0) - \inf_M \varphi_t + 1}              \\
	          &\qquad             
	          + A \left(\varphi_t (z) - \inf_M \varphi_t\right) - \frac{1}{\varphi_t (z) - \inf_M \varphi_t + 1} \\
	&\leq            
	          \ln |\nabla \varphi_t (z_0)|^2    -   A \varphi_t (z_0)   
	          + A \varphi_t (z) + 1 \\
	  &\leq
	  		\ln C(A,t) + A \left(\varphi_t (z) - \inf_M \varphi_t\right) + 1 .
\end{aligned}
\end{equation*}

\end{proof}

%
%
%

\medskip

\subsection{Second order estimate}

We have the following strong concavity due to \cite{GuanLiZhang2009}\cite{FangLaiMa}.
\begin{proposition}[Strong concavity]
For any $\bm{\lambda} \in \Gamma_n$ and $(\xi_1, \cdots ,\xi_n) \in \mathbb{C}^n$, it holds that
\begin{equation}
\label{concavity-4-36}
	\sum^n_{i = 1} \frac{S_{m - 1;i} (\bm{\lambda})}{\lambda_i} \xi_i \bar\xi_i  + \sum_{i,j} S_{m-2;ij} (\bm{\lambda}) \xi_i \bar\xi_j
	\geq \sum_{i,j} \frac{S_{m - 1;i} (\bm{\lambda} ) S_{m-1;j} (\bm{\lambda})}{S_m (\bm{\lambda})} \xi_i \bar\xi_j .
\end{equation}

\end{proposition}

By strong concavity~\eqref{concavity-4-36}, we obtain
\begin{equation}
\label{inequality-4-37}
\begin{aligned}
	\bar\partial_l \partial_l S_m (X^{-1})
	&\geq - \sum_i S_{m - 1;i} (X^{-1}) (X^{i\bar i})^2 X_{i\bar il\bar l} \\
	&\qquad + \sum_{i,j} S_{m - 1;i} (X^{-1}) (X^{i\bar i})^2 X^{j\bar j} X_{i\bar jl} X_{j\bar i\bar l} \\
	&\qquad +  \sum_{i,j} \frac{S_{m - 1;i} (X^{-1}) S_{m - 1;j} (X^{-1})}{S_m (X^{-1})} (X^{i\bar i})^2 (X^{j\bar j})^2 X_{i\bar i l} X_{j\bar j\bar l} .
\end{aligned}
\end{equation}

We introduce a quantity $w$ for $C^2$ estimate
where 
\begin{equation*}
 	w := \ln tr_\omega (\chi + t\chi + \tilde \chi + i\partial\bar\partial \varphi_t) .
\end{equation*}
\begin{lemma}
\label{lemma-4-3}

Let $\varphi_t \in C^2(M)$ be an admissible solution to   Equation~\eqref{approximation-equation-1-5} for $t > 1$. There are positive constants $N$ and $\theta$ such that when $w > N$ at a point $z \in M$,
\begin{equation*}
\begin{aligned}
	&\quad- \sum_i \left(\frac{c}{C^m_n}S_{n - m - 1;i} (X^{-1}) (X^{i\bar i})^2 + b_t f S_{n - 1;i} (X^{-1}) (X^{i\bar i})^2\right) \varphi_{t, i\bar i} \\
	&> 
	\theta \left( \sum_i \left(\frac{c}{C^m_n}S_{n - m - 1;i} (X^{-1}) (X^{i\bar i})^2 + b_t f S_{n - 1;i} (X^{-1}) (X^{i\bar i})^2\right)+ 1\right) ,
\end{aligned}
\end{equation*}
under coordinates around $z$ such that $\omega_{i\bar  j} = \delta_{ij}$ and $X_{i\bar j}$ is diagonal at $z$.
The constants $N$ and $\theta$ may depend on parameter $t$ and geometric data.

\end{lemma}

\begin{proof}

Without loss of generality, we may assume that $X_{1\bar 1} \geq \cdots \geq X_{n\bar n} > 0$.

By direct computation,
\begin{equation*}
\begin{aligned}
	&\quad  \sum_i \left(\frac{c}{C^m_n} S_{n - m - 1;i} (X^{-1}) (X^{i\bar i})^2 + b_t f S_{n - 1;i} (X^{-1}) (X^{i\bar i})^2 \right)    \\
	&=    \frac{c}{C^m_n} \left(S_{n - m} (X^{-1}) S_1 (X^{-1}) - (n - m + 1) S_{n - m + 1} (X^{-1}) \right) + b_t f S_n (X^{-1}) S_1 (X^{-1}) \\
	&\leq  \left( \frac{c}{C^m_n}  S_{n - m} (X^{-1})    + b_t f S_n (X^{-1}) \right) S_1 (X^{-1}) \\
	&=
	 S_1 (X^{-1}) .
\end{aligned}
\end{equation*}
Further, by Newton-Maclaurin inequality, 
\begin{equation*}
\begin{aligned}
	&\quad  \sum_i \left(\frac{c}{C^m_n} S_{n - m - 1;i} (X^{-1}) (X^{i\bar i})^2 + b_t f S_{n - 1;i} (X^{-1}) (X^{i\bar i})^2 \right)    \\
	&\geq \frac{c}{C^m_n} \frac{n - m}{n} S_{n - m} (X^{-1}) S_1 (X^{-1})  + b_t f S_n (X^{-1}) S_1 (X^{-1})  \\
	&\geq \frac{n - m}{n} S_1 (X^{-1})
	.
\end{aligned}
\end{equation*}
Then, we obtain
\begin{equation*}
\begin{aligned}
	&\quad - \sum_i \left(\frac{c}{C^m_n} S_{n - m - 1;i} (X^{-1}) (X^{i\bar i})^2 + b_t f S_{n - 1;i} (X^{-1}) (X^{i\bar i})^2 \right)  \varphi_{t,i\bar i}   \\
	&\geq  (1 + t) \delta_0 \sum_i \left(\frac{c}{C^m_n} S_{n - m - 1;i} (X^{-1}) (X^{i\bar i})^2 + b_t f S_{n - 1;i} (X^{-1}) (X^{i\bar i})^2 \right)   \\
		&\quad - \sum_i \left(\frac{c}{C^m_n} S_{n - m - 1;i} (X^{-1}) X^{i\bar i} + b_t f S_{n - 1;i} (X^{-1}) X^{i\bar i}  \right) \\
	&\geq (1 + t) \delta_0  \left(\frac{c}{C^m_n} \frac{n - m}{n} S_{n - m} (X^{-1}) S_1 (X^{-1})  + b_t f S_n (X^{-1}) S_1 (X^{-1}) \right)   \\
				&\quad -  \left(\frac{c (n - m)}{C^m_n} S_{n - m } (X^{-1}) +  n b_t f S_{n } (X^{-1})  \right) \\
	&\geq \frac{(1 + t) \delta_0}{n} \left(\frac{c (n - m)}{C^m_n} S_{n - m} (X^{-1})    + n b_t f S_n (X^{-1}) \right)  S_1 (X^{-1})  \\
				&\quad -   \left(\frac{c (n - m)}{C^m_n} S_{n - m } (X^{-1}) +  n b_t f S_{n } (X^{-1})  \right) 
				.
\end{aligned}
\end{equation*}
If $S_1(X^{-1}) \geq \frac{2 n }{ \delta_0}$, then
\begin{equation*}
\begin{aligned}
	&\quad - \sum_i \left(\frac{c}{C^m_n} S_{n - m - 1;i} (X^{-1}) (X^{i\bar i})^2 + b_t f S_{n - 1;i} (X^{-1}) (X^{i\bar i})^2 \right)  \varphi_{t,i\bar i}   \\
	&\geq \frac{(1 + t) \delta_0}{2 n} \left(\frac{c (n - m)}{C^m_n} S_{n - m} (X^{-1})    + n b_t f S_n (X^{-1}) \right)  S_1 (X^{-1})  \\
	&\geq \frac{(1 + t) \delta_0 (n - m)}{ 2n} S_1 (X^{-1}) \\
	&\geq
	 \frac{(1 + t) \delta_0 (n - m)}{ 2n}
	 \sum_i \left(\frac{c}{C^m_n} S_{n - m - 1;i} (X^{-1}) (X^{i\bar i})^2 + b_t f S_{n - 1;i} (X^{-1}) (X^{i\bar i})^2 \right)  
				.
\end{aligned}
\end{equation*}
Otherwise,
\begin{equation}
\label{case-4-42}
	X_{1\bar 1} \geq \cdots \geq X_{n \bar n} > \frac{\delta_0}{2 n} .
\end{equation}
In the case of \eqref{case-4-42}, we rewrite
\begin{equation}
\label{inequality-4-43}
\begin{aligned}
	&\quad - \sum_i \left(\frac{c}{C^m_n} S_{n - m - 1;i} (X^{-1}) (X^{i\bar i})^2 + b_t f S_{n - 1;i} (X^{-1}) (X^{i\bar i})^2 \right)  \varphi_{t,i\bar i}   \\
	&\geq  
	\frac{t}{2} \sum_i \left(\frac{c}{C^m_n} S_{n - m - 1;i} (X^{-1}) (X^{i\bar i})^2 + b_t f S_{n - 1;i} (X^{-1}) (X^{i\bar i})^2 \right)  \chi_{i\bar i} \\
	&\quad - \left(\frac{c}{C^m_n} S_{n - m - 1;1} (X^{-1}) (X^{1\bar 1})^2 + b_t f S_{n - 1;1} (X^{-1}) (X^{1\bar 1})^2 \right)  X_{1\bar 1} \\
	&\quad + \sum_i \left(\frac{c}{C^m_n} S_{n - m - 1;i} (X^{-1}) (X^{i\bar i})^2 + b_t f S_{n - 1;i} (X^{-1}) (X^{i\bar i})^2 \right) \left( \chi_{i\bar i} + \frac{t}{2} \chi_{i\bar i} + \tilde \chi_{i\bar i} \right) \\
	&\quad + \left(\frac{c}{C^m_n} S_{n - m - 1;1} (X^{-1}) (X^{1\bar 1})^2 + b_t f S_{n - 1;1} (X^{-1}) (X^{1\bar 1})^2 \right)  X_{1\bar 1} \\
	&\quad - \sum_i \left(\frac{c}{C^m_n} S_{n - m - 1;i} (X^{-1}) (X^{i\bar i})^2 + b_t f S_{n - 1;i} (X^{-1}) (X^{i\bar i})^2 \right)  X_{i\bar i} 
	.	
\end{aligned}
\end{equation}
We denote
\begin{equation*}
\tilde X := \chi  + \frac{t}{2} \chi + \tilde \chi + \sqrt{-1} X_{1\bar 1} d z^1 \wedge d\bar z^1 .
\end{equation*}
By concavity, it is from \eqref{inequality-4-43} and \eqref{case-4-42} that
\begin{equation}
\label{cone-4-43}
\begin{aligned}
	&\quad - \sum_i \left(\frac{c}{C^m_n} S_{n - m - 1;i} (X^{-1}) (X^{i\bar i})^2 + b_t f S_{n - 1;i} (X^{-1}) (X^{i\bar i})^2 \right)  \varphi_{t,i\bar i}   \\
	&\geq \frac{t}{2} \sum_i \left(\frac{c}{C^m_n} S_{n - m - 1;i} (X^{-1}) (X^{i\bar i})^2 + b_t f S_{n - 1;i} (X^{-1}) (X^{i\bar i})^2 \right)  \chi_{i\bar i} \\
		&\quad - \left(\frac{c}{C^m_n} S_{n - m - 1;1} (X^{-1})   + b_t f S_{n - 1;1} (X^{-1})  \right) X^{1\bar 1}\\
		&\quad - \left(\frac{c}{C^m_n} S_{n - m} (\tilde X^{-1}) + b_t f S_n (\tilde X^{-1})\right) + 1 \\
	&> 
	\frac{t}{2} \sum_i \left(\frac{c}{C^m_n} S_{n - m - 1;i} (X^{-1}) (X^{i\bar i})^2 + b_t f S_{n - 1;i} (X^{-1}) (X^{i\bar i})^2 \right)  \chi_{i\bar i} \\
		&\quad - \left(\frac{c}{C^m_n} C^{n - m - 1}_{n - 1}  \left(\frac{2n}{\delta_0}\right)^{n - m - 1} + b_t f  \left(\frac{2n}{\delta_0}\right)^{n - 1}  \right) X^{1\bar 1}\\
		&\quad - \left(\frac{c}{C^m_n} S_{n - m} (\tilde X^{-1}) + b_t f S_n (\tilde X^{-1})\right) + 1 
		.
\end{aligned}
\end{equation}
Similar to \eqref{inequality-4-19},
\begin{equation}
\label{inequality-4-47-1}
\begin{aligned}
	&\quad \tilde X^n - \left(1 + \frac{t}{4}\right)^{n - m} \left(c \tilde X^m \wedge \omega^{n - m} + b_t f \omega^n\right) \\
	&\geq \left(1 + \frac{t}{4}\right)^n \left(\frac{\hat\chi^n}{(1 + \frac{t}{4})^n} + \chi^n - c \chi^m \wedge \omega^{n - m} - \frac{b_t}{\left(1 + \frac{t}{4}\right)^m} f \omega^n\right) \\
	&> 
	n \sqrt{-1} X_{1\bar 1} d z^1 \wedge d\bar z^1 \wedge \left(\frac{t}{4} \chi + \tilde \chi\right)^{n - 1} \\
	&\qquad -  \left(1 + \frac{t}{4}\right)^n  c \chi^m \wedge \omega^{n - m}  - \left(1 + \frac{t}{4}\right)^{n - m} b_t f \omega^n 
	,
\end{aligned}
\end{equation}
where
\begin{equation*}
	\hat\chi := \frac{t}{4} \chi + \tilde \chi + \sqrt{-1} X_{1\bar 1} dz^1 \wedge d\bar z^1 .
\end{equation*}
If
\begin{equation*}
	X_{1\bar 1} \geq \frac{4^{n-1}}{t^{n - 1} \delta^{n - 1}_0 } \left(\left(1 + \frac{t}{4}\right)^n \frac{c}{C^m_n} S_m (\chi) + \left(1 + \frac{t}{4}\right)^{n - m} b_t f\right) ,
\end{equation*}
then Inequality~\eqref{inequality-4-47-1} implies that
\begin{equation}
\label{cone-4-47}
\tilde X^n - \left(1 + \frac{t}{4}\right)^{n - m} \left(c \tilde X^m \wedge \omega^{n - m} + b_t f \omega^n\right) > 0 .
\end{equation}
Inequality~\eqref{cone-4-47} tells us that
\begin{equation}
\label{cone-4-48}
	\frac{c}{C^m_n} S_{n - m} (\tilde X^{-1})+ b_t f S_n (\tilde X^{-1})< \frac{1}{\left(1 + \frac{t}{4}\right)^{n- m}} .
\end{equation}
Substituting \eqref{cone-4-48} into \eqref{cone-4-43},
\begin{equation*}
\begin{aligned}
	&\quad - \sum_i \left(\frac{c}{C^m_n} S_{n - m - 1;i} (X^{-1}) (X^{i\bar i})^2 + b_t f S_{n - 1;i} (X^{-1}) (X^{i\bar i})^2 \right)  \varphi_{t,i\bar i}   \\
	&> 
		\frac{t}{2} \sum_i \left(\frac{c}{C^m_n} S_{n - m - 1;i} (X^{-1}) (X^{i\bar i})^2 + b_t f S_{n - 1;i} (X^{-1}) (X^{i\bar i})^2 \right)  \chi_{i\bar i} \\
			&\quad - \left(\frac{(n - m)c}{n}  \left(\frac{2n}{\delta_0}\right)^{n - m - 1} + b_t f  \left(\frac{2n}{\delta_0}\right)^{n - 1}  \right) \frac{1}{X_{1\bar 1}} + \frac{(n - m) t}{4 \left(1 + \frac{t}{4}\right)^{n - m}}  
	.
\end{aligned}
\end{equation*}
Therefore, 
\begin{equation*}
\begin{aligned}
	&\quad - \sum_i \left(\frac{c}{C^m_n} S_{n - m - 1;i} (X^{-1}) (X^{i\bar i})^2 + b_t f S_{n - 1;i} (X^{-1}) (X^{i\bar i})^2 \right)  \varphi_{t,i\bar i}   \\
	&> 
		\frac{\delta_0 t}{2} \sum_i \left(\frac{c}{C^m_n} S_{n - m - 1;i} (X^{-1}) (X^{i\bar i})^2 + b_t f S_{n - 1;i} (X^{-1}) (X^{i\bar i})^2 \right)   
		 + \frac{(n - m) t}{8 \left(1 + \frac{t}{4}\right)^{n - m}}  
	,
\end{aligned}
\end{equation*}
when
\begin{equation*}
\begin{aligned}
	X_{1\bar 1} &\geq \frac{4^{n-1}}{t^{n - 1}\delta^{n - 1}_0}  \left(\left(1 + \frac{t}{4}\right)^n \frac{c}{C^m_n} S_m (\chi) + \left(1 + \frac{t}{4}\right)^{n - m} b_t f\right) \\
	&\qquad + \frac{8}{(n - m) t} \left(1 + \frac{t}{4}\right)^{n - m} \left(\frac{(n - m)c}{n}  \left(\frac{2n}{\delta_0}\right)^{n - m - 1} + b_t f  \left(\frac{2n}{\delta_0}\right)^{n - 1}  \right) .
\end{aligned}
\end{equation*}

\end{proof}

Now we shall prove second order estimate.
\begin{theorem}[Second order estimate]
\label{theorem-C2-estimate}

Let $\varphi_t \in C^4(M)$ be an admissible solution to Equation~\eqref{approximation-equation-1-5} for $t > 0$. Then there are positive constants $C$ and $A$ such that
\begin{equation*}
\sup_M w \leq C e^{A \left(\varphi_t - \inf_M \varphi_t\right)} .
\end{equation*}
The constants $C$ and $A$ may depend on parameter $t$ and given geometric data.

\end{theorem}

\begin{proof}

We consider the function
\begin{equation}
\label{test-function-4-41}
		\ln w - A \varphi_t  ,
\end{equation}
where $A >> 1$ is to be determined later. There is a point $z_0$ where $\ln w - A \varphi_t$ reaches its maximal value. We pick a local chart around $z_0$ such that $\omega_{i\bar  j} = \delta_{ij}$ and $X_{i\bar j}$ is diagonal at $z_0$.

Differentiating equation form~\eqref{equation-4-3} at $z_0$, we obtain
\begin{equation}
\begin{aligned}
	0
	&= 
	- \frac{c}{C^m_n} \sum_i S_{n - m - 1;i} (X^{-1}) (X^{i\bar i})^2 X_{i\bar i l} + b_t f_l S_n (X^{-1}) - b_t S_n (X^{-1}) f \sum_i X^{i\bar i} X_{i\bar il} ,
\end{aligned}
\end{equation}
and
\begin{equation}
\label{derivative-4-53}
\begin{aligned}
	0
	&\geq
	\frac{c}{C^m_n} \sum_{i,j} S_{n - m - 1;i} (X^{-1}) (X^{i\bar i})^2 X^{j\bar j} X_{j\bar i\bar l} X_{i\bar jl} - \frac{c}{C^m_n} \sum_i S_{n - m - 1;i} (X^{-1}) (X^{i\bar i})^2 X_{i\bar il\bar l} \\
	&\qquad + b_t f \sum_{i,j} S_{n - 1;i} (X^{-1}) (X^{i\bar i})^2 X^{j\bar j} X_{j\bar i\bar l} X_{i\bar jl} - b_t f \sum_i S_{n  - 1;i} (X^{-1}) (X^{i\bar i})^2 X_{i\bar il\bar l} \\
	&\qquad + b_t f \frac{1}{S_n (X^{-1} )} \partial_l S_n (X^{-1}) \bar\partial _l S_n (X^{-1}) \\
	&\qquad + b_t f_{l\bar l} S_n (X^{-1}) + b_t f_l \bar\partial_l S_n (X^{-1})+ b_t f_{\bar l} \partial_l S_n (X^{-1}) \\
	&\geq 
 		\left(\frac{c}{C^m_n} \sum_{i,j} S_{n - m - 1;i} (X^{-1})  + b_t f \sum_{i,j} S_{n - 1;i} (X^{-1}) \right) (X^{i\bar i})^2 X^{j\bar j} X_{j\bar i\bar l} X_{i\bar jl} \\
 	&\qquad - \left(\frac{c}{C^m_n} \sum_{i,j} S_{n - m - 1;i} (X^{-1})  + b_t f \sum_{i,j} S_{n - 1;i} (X^{-1}) \right) (X^{i\bar i})^2 X_{i\bar il\bar l} \\
	&\qquad + b_t f_{l\bar l} S_n (X^{-1})   - b_t \frac{ f_l f_{\bar l} }{f} S_n (X^{-1}) 
	.
\end{aligned}
\end{equation}
The first inequality in \eqref{derivative-4-53} is derived via Inequality~\eqref{inequality-4-37}, while the second is by Cauchy-Schwarz inequality. 
In sum,
\begin{equation}
\label{inequality-4-40}
\begin{aligned}
	&\quad \sum_i  \left(\frac{c}{C^m_n} S_{n - m - 1;i} (X^{-1}) (X^{i\bar i})^2 + b_t f S_{n - 1;i} (X^{-1}) (X^{i\bar i})^2\right) X_{i\bar il\bar l} \\
	&\geq \sum_{i,j} \left(	\frac{c}{C^m_n}  S_{n - m - 1;i} (X^{-1}) (X^{i\bar i})^2    + b_t f  S_{n - 1;i} (X^{-1}) (X^{i\bar i})^2 \right) X^{j\bar j} X_{j\bar i\bar l} X_{i\bar jl}  \\
		&\qquad - b_t  S_n (X^{-1})  \frac{f_l f_{\bar l}}{f}  + b_t  S_n (X^{-1})  f_{l\bar l} .
\end{aligned}
\end{equation}
Moreover, we know that at $z_0$ 
\begin{equation}
	0 = \sum_l X_{l\bar li} - w A \varphi_{t,i}  ,
\end{equation}
and
\begin{equation}
\label{inequality-4-44}
\begin{aligned}
	0 &\geq 
	\sum_l X_{l\bar li\bar i} - \frac{w_i w_{\bar i}}{w} - w A \varphi_{t,i\bar i} \\
	&= \sum_l \left(X_{i\bar il\bar l} - R_{l\bar l i\bar i} X_{i\bar i} + R_{i\bar il\bar l} X_{l\bar l} + G_{i\bar il\bar l} \right) - \frac{w_i w_{\bar i}}{w} - w A \varphi_{t,i\bar i}   .
\end{aligned}
\end{equation}
where
\begin{equation}
\begin{aligned}
	G_{i\bar il\bar l} 
	&= 
	(\chi  + \tilde \chi)_{l\bar li\bar i} - (\chi   + \tilde \chi)_{i\bar il\bar l} + \sum_k R_{l\bar li\bar k} (\chi +  \tilde \chi)_{k\bar i} - \sum_k R_{i\bar il\bar k} (\chi  + \tilde \chi)_{k\bar l} \\
	&\qquad +
	 t\chi_{l\bar li\bar i} -  t\chi_{i\bar il\bar l} + t \sum_k R_{l\bar li\bar k} \chi_{k\bar i} - t \sum_k R_{i\bar il\bar k}  \chi_{k\bar l}  .
\end{aligned}
\end{equation}
Multiplying \eqref{inequality-4-44} by $ \frac{c}{C^m_n} S_{n - m - 1;i} (X^{-1}) (X^{i\bar i})^2 + b_t f S_{n - 1;i} (X^{-1}) (X^{i\bar i})^2 $,
\begin{equation}
\label{inequality-4-46}
\begin{aligned}
	0 &\geq 
	\sum_{i,l} \left(\frac{c}{C^m_n} S_{n - m - 1;i} (X^{-1}) (X^{i\bar i})^2 + b_t f S_{n - 1;i} (X^{-1}) (X^{i\bar i})^2 \right)  X_{i\bar il\bar l}   \\
	&\qquad - C (w + 1)	\sum_{i} \left(\frac{c}{C^m_n} S_{n - m - 1;i} (X^{-1}) (X^{i\bar i})^2 + b_t f S_{n - 1;i} (X^{-1}) (X^{i\bar i})^2 \right)  \\
	&\qquad - \sum_i \left(\frac{c}{C^m_n} S_{n - m - 1;i} (X^{-1}) (X^{i\bar i})^2 + b_t f S_{n - 1;i} (X^{-1}) (X^{i\bar i})^2 \right)  \frac{w_i w_{\bar i}}{w} \\
	&\qquad - w A \sum_i \left(\frac{c}{C^m_n} S_{n - m - 1;i} (X^{-1}) (X^{i\bar i})^2 + b_t f S_{n - 1;i} (X^{-1}) (X^{i\bar i})^2 \right)  \varphi_{t,i\bar i}   .
\end{aligned}
\end{equation}

By direct calculation,
\begin{equation}
\begin{aligned}
	\left| X_{i\bar j l} - X_{j\bar l} \frac{w_i}{w}\right|^2
	&= X_{i\bar jl} X_{j\bar i\bar l} - 2 \mathfrak{Re} \left\{ X_{i\bar jl} X_{l\bar j} \frac{w_{\bar i}}{w}\right\} + X_{j\bar l} X_{l\bar j} \frac{  w_i   w_{\bar i}}{w^2} ,
\end{aligned}
\end{equation}
and hence
\begin{equation}
\label{inequality-4-47}
\begin{aligned}
	\sum_l X_{i\bar jl} X_{j\bar i\bar l}  &\geq 2 \sum_l \mathfrak{Re} \left\{ X_{i\bar jl} X_{l\bar j} \frac{  w_{\bar i}}{w}\right\} - \sum_l X_{j\bar l} X_{l\bar j} \frac{  w_i  w_{\bar i}}{w^2} \\
	&= 2 X_{j\bar j} \mathfrak{Re} \left\{ X_{i\bar jj}   \frac{  w_{\bar i}}{w}\right\} - X^2_{j\bar j}  \frac{  w_i   w_{\bar i}}{w^2} \\
	&= 2 X_{j\bar j} \mathfrak{Re} \left\{ X_{j\bar ji}   \frac{  w_{\bar i}}{w}\right\} - X^2_{j\bar j}  \frac{  w_i   w_{\bar i}}{w^2}.
\end{aligned}
\end{equation}
Multiplying \eqref{inequality-4-47} by $\left(\frac{c}{C^m_n} S_{n - m - 1;i} (X^{-1}) (X^{i\bar i})^2 + b_t f S_{n - 1;i} (X^{-1}) (X^{i\bar i})^2\right) X^{j\bar j} $ and summing them up,
\begin{equation}
\label{inequality-4-48}
\begin{aligned}
	&\quad \sum_{i,j,l}  \left(\frac{c}{C^m_n} S_{n - m - 1;i} (X^{-1}) (X^{i\bar i})^2 + b_t f S_{n - 1;i} (X^{-1}) (X^{i\bar i})^2\right) X^{j\bar j} X_{i\bar jl} X_{j\bar i\bar l} \\
	&\geq \sum_{i,j}  \left(\frac{c}{C^m_n} S_{n - m - 1;i} (X^{-1}) (X^{i\bar i})^2 + b_t f S_{n - 1;i} (X^{-1}) (X^{i\bar i})^2\right) X^{j\bar j} \\
	&\qquad\qquad \cdot \left(2 X_{j\bar j} \mathfrak{Re} \left\{ X_{j\bar ji}   \frac{ w_{\bar i}}{w}\right\} - X^2_{j\bar j}  \frac{  w_i   w_{\bar i}}{w^2} \right) \\
	&= \sum_{i}  \left(\frac{c}{C^m_n} S_{n - m - 1;i} (X^{-1}) (X^{i\bar i})^2 + b_t f S_{n - 1;i} (X^{-1}) (X^{i\bar i})^2\right)   \frac{  w_i   w_{\bar i}}{w}   .
\end{aligned}
\end{equation}
Substituting \eqref{inequality-4-48} into \eqref{inequality-4-40},
\begin{equation}
\label{inequality-4-50}
\begin{aligned}
	&\qquad \sum_{i,l}  \left(\frac{c}{C^m_n} S_{n - m - 1;i} (X^{-1}) (X^{i\bar i})^2 + b_t f S_{n - 1;i} (X^{-1}) (X^{i\bar i})^2\right) X_{i\bar il\bar l} \\
	&\geq \sum_{i}  \left(\frac{c}{C^m_n} S_{n - m - 1;i} (X^{-1}) (X^{i\bar i})^2 + b_t f S_{n - 1;i} (X^{-1}) (X^{i\bar i})^2\right)    \frac{  w_i   w_{\bar i}}{w} \\
		&\qquad - b_t  S_n (X^{-1}) \sum_l \frac{f_l f_{\bar l}}{f}  + b_t  S_n (X^{-1}) \sum_l  f_{l\bar l} .
\end{aligned}
\end{equation}
Substituting \eqref{inequality-4-50} into \eqref{inequality-4-46},
\begin{equation}
\begin{aligned}
	0 
	&\geq  - b_t  S_n (X^{-1}) \sum_l \frac{f_l f_{\bar l}}{f}  + b_t  S_n (X^{-1}) \sum_l  f_{l\bar l} \\
	&\qquad - C (w + 1)	\sum_{i} \left(\frac{c}{C^m_n} S_{n - m - 1;i} (X^{-1}) (X^{i\bar i})^2 + b_t f S_{n - 1;i} (X^{-1}) (X^{i\bar i})^2 \right)  \\
	&\qquad - w A \sum_i \left(\frac{c}{C^m_n} S_{n - m - 1;i} (X^{-1}) (X^{i\bar i})^2 + b_t f S_{n - 1;i} (X^{-1}) (X^{i\bar i})^2 \right)  \varphi_{t,i\bar i}   
	.
\end{aligned}
\end{equation}
According to Lemma~\ref{lemma-4-3}, there are positive constants $N(t)$ and $\theta(t)$ such that 
\begin{equation}
\label{inequality-4-57}
\begin{aligned}
	0 
	&\geq  - b_t  f S_n (X^{-1}) \left|\sum_l \left(\frac{f_l f_{\bar l}}{f^2}  - \frac{f_{l\bar l}}{f} \right)\right| \\
	&\qquad - C (w + 1)	\sum_{i} \left(\frac{c}{C^m_n} S_{n - m - 1;i} (X^{-1}) (X^{i\bar i})^2 + b_t f S_{n - 1;i} (X^{-1}) (X^{i\bar i})^2 \right)  \\
	&\qquad + w A \theta (t) \left(\sum_i \left(\frac{c}{C^m_n} S_{n - m - 1;i} (X^{-1}) (X^{i\bar i})^2 + b_t f S_{n - 1;i} (X^{-1}) (X^{i\bar i})^2 \right)  + 1 \right) \\
	&\geq 
	- C- C (w + 1)	\sum_{i} \left(\frac{c}{C^m_n} S_{n - m - 1;i} (X^{-1}) (X^{i\bar i})^2 + b_t f S_{n - 1;i} (X^{-1}) (X^{i\bar i})^2 \right)  \\
		&\qquad + w A \theta (t) \left(\sum_i \left(\frac{c}{C^m_n} S_{n - m - 1;i} (X^{-1}) (X^{i\bar i})^2 + b_t f S_{n - 1;i} (X^{-1}) (X^{i\bar i})^2 \right)  + 1 \right) 
	.
\end{aligned}
\end{equation}
when $w > N(t)$. If we choose $A$ sufficiently large, Inequality~\eqref{inequality-4-57} is impossible.

Therefore, at any $z \in M$
\begin{equation*}
	\ln w -  A \varphi_t \leq \ln w (z_0) -  A \varphi_t (z_0) \leq \ln N(t) - A \inf_M \varphi_t  .
\end{equation*}

\end{proof}

\medskip
\subsection{Higher order estimates and continuity method}

Higher order estimates can be achieved by Evans-Krylov theory~\cite{Evans1982}\cite{Krylov1982}\cite{TosattiWangWeinkoveYang2015} and Schauder estimate. To solve Equation~\eqref{approximation-equation-1-5}, we can apply variant continuity methods introduced by the author~\cite{Sun2016} and Sz\'ekelyhidi~\cite{Szekelyhidi}.

\medskip

\section{Uniform Smoothness}
\label{smoothness}

The estimates in Section~\ref{solvability} may blow up as $t$ approaches $0$. In this section, we shall study uniform smoothness for approximation equation~\eqref{approximation-equation-1-5}. Since the $L^\infty$ estimate proven in Section~\ref{L-estimate} is $t$-independent and higher order estimates are local, it is sufficient to show that $C^2$ estimate is also $t$-independent in certain sense.

%

%
%
%
%
%

First of all, we need to modify Lemma~\ref{lemma-4-3}.
Thanks to the work of Boucksom~\cite{Boucksom2004}, we know that there is a function $\rho$ and a constant $\delta > 0$ such that 
$\rho$ is smooth in the ample locus $Amp (\tilde\chi)$ with analytic singularities,  
and $\tilde\chi + \sqrt{-1}\partial\bar\partial\rho \geq \delta \chi $.
\begin{lemma}
\label{lemma-5-2}

Let $\varphi_t \in C^2(M)$ be an admissible solution to   Equation~\eqref{approximation-equation-1-5} for $t > 0$. There are positive constants $N$ and $\theta$ such that when $w > N$ at a point $z \in Amp (\tilde \chi)$,
\begin{equation*}
\begin{aligned}
	&\quad \sum_i \left(\frac{c}{C^m_n}S_{n - m - 1;i} (X^{-1}) (X^{i\bar i})^2 + b_t f S_{n - 1;i} (X^{-1}) (X^{i\bar i})^2\right) \left(\rho_{i\bar i} - \varphi_{t, i\bar i} \right) \\
	&> 
	\theta \left( \sum_i \left(\frac{c}{C^m_n}S_{n - m - 1;i} (X^{-1}) (X^{i\bar i})^2 + b_t f S_{n - 1;i} (X^{-1}) (X^{i\bar i})^2\right)+ 1\right) ,
\end{aligned}
\end{equation*}
under coordinates around $z$ such that $\omega_{i\bar  j} = \delta_{ij}$ and $X_{i\bar j}$ is diagonal at $z$.
The constants $N$ and $\theta$ are independent of parameter $t$.

\end{lemma}

\begin{proof}

The proof is simply a small modification of that for Lemma~\ref{lemma-4-3}. To clarify the dependence for constants, we  do the argument once more with some intermediate steps skipped. 
We refer the readers to the proof of Lemma~\ref{lemma-4-3} for omitted details.

We have
\begin{equation*}
\begin{aligned}
\frac{n - m}{n} S_1 (X^{-1})
&\leq 
\sum_i \left(\frac{c}{C^m_n} S_{n - m - 1;i} (X^{-1}) (X^{i\bar i})^2 + b_t f S_{n - 1;i} (X^{-1}) (X^{i\bar i})^2 \right) \\
&\leq S_1 (X^{-1}) .
\end{aligned}
\end{equation*}
Then direct calculation shows that
\begin{equation}
\begin{aligned}
	&\quad  \sum_i \left(\frac{c}{C^m_n} S_{n - m - 1;i} (X^{-1}) (X^{i\bar i})^2 + b_t f S_{n - 1;i} (X^{-1}) (X^{i\bar i})^2 \right) \left(\rho_{i\bar i} - \varphi_{t,i\bar i} \right)   \\
	&\geq \left( 1 + t + \delta\right) \delta_0 \sum_i \left(\frac{c}{C^m_n} S_{n - m - 1;i} (X^{-1}) (X^{i\bar i})^2 + b_t f S_{n - 1;i} (X^{-1}) (X^{i\bar i})^2 \right)   \\
	&\geq \frac{ \left(    1    + \delta \right)  \delta_0}{n} \left(\frac{c (n - m)}{C^m_n} S_{n - m} (X^{-1})    + n b_t f S_n (X^{-1}) \right)  S_1 (X^{-1})  \\
	&\qquad -   \left(\frac{c (n - m)}{C^m_n} S_{n - m } (X^{-1}) +  n b_t f S_{n } (X^{-1})  \right) 				
				.
\end{aligned}
\end{equation}
If $S_1(X^{-1}) \geq \frac{2 n }{(1 + \delta) \delta_0}$, then
\begin{equation}
\begin{aligned}
	&\quad  \sum_i \left(\frac{c}{C^m_n} S_{n - m - 1;i} (X^{-1}) (X^{i\bar i})^2 + b_t f S_{n - 1;i} (X^{-1}) (X^{i\bar i})^2 \right) \left(\rho_{i\bar i} - \varphi_{t,i\bar i}   \right)\\
	&\geq
	 \frac{(1 + \delta) \delta_0 (n - m)}{ 2n}
	 \sum_i \left(\frac{c}{C^m_n} S_{n - m - 1;i} (X^{-1}) (X^{i\bar i})^2 + b_t f S_{n - 1;i} (X^{-1}) (X^{i\bar i})^2 \right)  
				.
\end{aligned}
\end{equation}
Otherwise, we have
\begin{equation}
\label{case-5-4}
	X_{1\bar 1} \geq \cdots \geq X_{n \bar n} > \frac{(1 + \delta) \delta_0}{2 n} .
\end{equation}
In the case of \eqref{case-5-4}, we obtain as in \eqref{cone-4-43},
\begin{equation}
\label{cone-5-7}
\begin{aligned}
	&\quad  \sum_i \left(\frac{c}{C^m_n} S_{n - m - 1;i} (X^{-1}) (X^{i\bar i})^2 + b_t f S_{n - 1;i} (X^{-1}) (X^{i\bar i})^2 \right)  \left(\rho_{i\bar i} - \varphi_{t,i\bar i} \right)  \\
	&> 
	\frac{\delta}{2} \sum_i \left(\frac{c}{C^m_n} S_{n - m - 1;i} (X^{-1}) (X^{i\bar i})^2 + b_t f S_{n - 1;i} (X^{-1}) (X^{i\bar i})^2 \right) \chi_{i\bar i}  \\
		&\quad - \left(\frac{c}{C^m_n} C^{n - m - 1}_{n - 1}  \left(\frac{2n}{(1 + \delta)\delta_0  }\right)^{n - m - 1} + b_t f  \left(\frac{2n}{(1 + \delta) \delta_0  }\right)^{n - 1}  \right) X^{1\bar 1}\\
		&\quad - \left(\frac{c}{C^m_n} S_{n - m} (\tilde X^{-1}) + b_t f S_n (\tilde X^{-1})\right) + 1 
		,
\end{aligned}
\end{equation}
where
\begin{equation}
\tilde X := \chi  + \frac{\delta}{2} \chi + \sqrt{-1} X_{1\bar 1} d z^1 \wedge d\bar z^1 .
\end{equation}
Similar to \eqref{inequality-4-47-1},
\begin{equation}
\label{inequality-5-7}
\begin{aligned}
	&\quad \tilde X^n - \left(1 + \frac{\delta}{4}\right)^{n - m} \left(c \tilde X^m \wedge \omega^{n - m} + b_t f \omega^n\right) \\
	&> 
	n \sqrt{-1} X_{1\bar 1} d z^1 \wedge d\bar z^1 \wedge \left(\frac{\delta}{4} \chi  \right)^{n - 1} -  \left(1 + \frac{\delta}{4}\right)^n  c \chi^m \wedge \omega^{n - m}  - \left(1 + \frac{\delta}{4}\right)^{n - m} b_t f \omega^n 
	.
\end{aligned}
\end{equation}
If
\begin{equation*}
	X_{1\bar 1} \geq \frac{4^{n-1}}{\delta^{n - 1}\delta^{n - 1}_0}  \left(\left(1 + \frac{\delta}{4}\right)^n \frac{c}{C^m_n} S_m (\chi) + \left(1 + \frac{\delta}{4}\right)^{n - m} b_t f\right) ,
\end{equation*}
then \eqref{inequality-5-7} implies that
\begin{equation*}
\label{cone-5-11}
\tilde X^n - \left(1 + \frac{t}{4}\right)^{n - m} \left(c \tilde X^m \wedge \omega^{n - m} + b_t f \omega^n\right) > 0 ,
\end{equation*}
and hence
\begin{equation}
\label{cone-5-12}
	\frac{c}{C^m_n} S_{n - m} (\tilde X^{-1})+ b_t f S_n (\tilde X^{-1})< \frac{1}{\left(1 + \frac{\delta}{4}\right)^{n- m}} .
\end{equation}
Substituting \eqref{cone-5-12} into \eqref{cone-5-7},
\begin{equation*}
\begin{aligned}
	&\quad - \sum_i \left(\frac{c}{C^m_n} S_{n - m - 1;i} (X^{-1}) (X^{i\bar i})^2 + b_t f S_{n - 1;i} (X^{-1}) (X^{i\bar i})^2 \right)  \varphi_{t,i\bar i}   \\
	&> 
		\frac{\delta}{2} \sum_i \left(\frac{c}{C^m_n} S_{n - m - 1;i} (X^{-1}) (X^{i\bar i})^2 + b_t f S_{n - 1;i} (X^{-1}) (X^{i\bar i})^2 \right)  \chi_{i\bar i} \\
			&\quad - \left(\frac{(n - m)c}{n}  \left(\frac{2n}{(1 + \delta)\delta_0}\right)^{n - m - 1} + b_t f  \left(\frac{2n}{(1 + \delta)\delta_0}\right)^{n - 1}  \right) \frac{1}{X_{1\bar 1}} + \frac{(n - m) \delta}{4 \left(1 + \frac{\delta}{4}\right)^{n - m}}  
	.
\end{aligned}
\end{equation*}
Therefore, 
\begin{equation*}
\begin{aligned}
	&\quad  \sum_i \left(\frac{c}{C^m_n} S_{n - m - 1;i} (X^{-1}) (X^{i\bar i})^2 + b_t f S_{n - 1;i} (X^{-1}) (X^{i\bar i})^2 \right) (\rho_{i\bar i} - \varphi_{t,i\bar i}  ) \\
	&> 
		\frac{\delta_0 \delta}{2} \sum_i \left(\frac{c}{C^m_n} S_{n - m - 1;i} (X^{-1}) (X^{i\bar i})^2 + b_t f S_{n - 1;i} (X^{-1}) (X^{i\bar i})^2 \right)   
		 + \frac{(n - m) \delta}{8 \left(1 + \frac{\delta}{4}\right)^{n - m}}  
	.
\end{aligned}
\end{equation*}
when
\begin{equation*}
\begin{aligned}
	&X_{1\bar 1} \geq \frac{4^{n-1}}{\delta^{n - 1}\delta^{n - 1}_0}  \left(\left(1 + \frac{\delta}{4}\right)^n \frac{c}{C^m_n} S_m (\chi) + \left(1 + \frac{\delta}{4}\right)^{n - m} b_1 f\right) \\
	&+ \frac{8}{(n - m) \delta} \left(1 + \frac{\delta}{4}\right)^{n - m} \left(\frac{(n - m)c}{n}  \left(\frac{2n}{(1 + \delta)\delta_0}\right)^{n - m - 1} +  \left(\frac{2n}{(1 + \delta)\delta_0}\right)^{n - 1} b_1 f  \right) .
\end{aligned}
\end{equation*}

\end{proof}

In the proof of Theorem~\ref{theorem-C2-estimate}, we may instead consider the function
\begin{equation*}
		\ln w + A (\rho - \varphi_t)  ,
\end{equation*}
where $A$ is to be specified. Function $\ln w + A (\rho - \varphi_t) $ attains its maximal value at some point in the ample locus of $\tilde\chi$. Replacing Lemma~\ref{lemma-4-3} by Lemma~\ref{lemma-5-2} in the proof of Theorem~\ref{theorem-C2-estimate}, we shall obtain a $C^2$ estimate independent of parameter $t$.
\begin{theorem}[Second order estimate independent of $t$]

Let $\varphi_t \in C^4(M)$ be an admissible solution to Equation~\eqref{approximation-equation-1-5} for $t > 0$. Then there are positive constants $C$ and $A$ such that
\begin{equation*}
	w \leq C e^{A \left(-\rho + \varphi_t - \inf_M \varphi_t\right)} .
\end{equation*}
In particular, constants $C$ and $A$ are independent of parameter $t$.
\end{theorem}

\medskip
\section{Stability estimate}
\label{stability}

Assume that
\begin{equation*}
	(\chi + t\chi + \tilde \chi + \sqrt{-1}\partial\bar\partial \varphi_1)^n = c (\chi + t\chi + \tilde \chi + \sqrt{-1}\partial\bar\partial \varphi_1)^m \wedge \omega^{n - m} + b_t f_1 \omega^n , \quad \sup_M \varphi_1 = 0 ,
\end{equation*}
and
\begin{equation*}
	(\chi + t\chi + \tilde \chi + \sqrt{-1}\partial\bar\partial \varphi_2)^n = c (\chi + t\chi + \tilde \chi + \sqrt{-1} \partial\bar\partial \varphi_2)^m \wedge \omega^{n - m} + b_t f_2 \omega^n
	.
\end{equation*}
In this section, we shall study the difference  of $\varphi_1$ and $\varphi_2$ in $L^\infty$ norm. 
%
%
%
%
%
%
\begin{theorem}
\label{theorem-6-1}

For $K > 0$, if $\varphi_1$ and $\varphi_2$ satisfy $ \Vert (\varphi_2 - \varphi_1)^+ \Vert_{L^{q^*}}  \leq K$, then we have
\begin{equation*}
	\sup_M (\varphi_2 - \varphi_1)
	\leq 
	C (\chi, \tilde \chi, \omega , K ,  q, \Vert f_1 \Vert_{L^q}) \Vert (\varphi_2 - \varphi_1)^+ \Vert^{\frac{1}{n + 1}}_{L^{q^*}} .
\end{equation*}
where $q^* = \frac{q}{q - 1}$.

\end{theorem}
\begin{proof}

According to the results in Section~\ref{solvability}, we can solve
\begin{equation*}
	\left(\chi + \frac{t}{2} \chi + \frac{1}{2} \tilde \chi + \sqrt{-1} \partial\bar\partial v_t\right)^n  
	= c \left(\chi + \frac{t}{2} \chi + \frac{1}{2} \tilde \chi + \sqrt{-1}\partial\bar\partial v_t \right)^m \wedge \omega^{n - m} + c_t \omega^n , 
\end{equation*}
where $\sup_M v_t = 0$ and
\begin{equation*}
	\int_M \left(\chi + \frac{t}{2} \chi + \frac{1}{2} \tilde \chi\right)^n = c\int_M \left(\chi + \frac{t}{2} \chi + \frac{1}{2} \tilde \chi\right)^m \wedge \omega^{n - m} + c_t \int_M \omega^n .
\end{equation*}
By Theorem~\ref{theorem-3-2}, it is easy to see that $v_t$ is uniformly bounded (independent of parameter $t$) and $c_t \geq 0$. As a result, we have
\begin{equation*}
	\left(\chi + \frac{t}{2} \chi + \frac{1}{2} \tilde \chi + \sqrt{-1} \partial\bar\partial v_t\right)^n  
	\geq
	c \left(\chi + \frac{t}{2} \chi + \frac{1}{2} \tilde \chi + \sqrt{-1}\partial\bar\partial v_t \right)^m \wedge \omega^{n - m}  , 
\end{equation*}
and then  by concavity of $\frac{S_n}{S_m}$,
\begin{equation}
\label{inequality-5-1-1}
\begin{aligned}
	& \left( (1 - r) (\chi + t\chi + \tilde \chi + \sqrt{-1} \partial\bar\partial \varphi_2) + r \left(\chi + \dfrac{t}{2} \chi + \dfrac{1}{2} \tilde \chi + \sqrt{-1} \partial\bar\partial v_t\right) \right)^n \\
	\geq& c \left( (1 - r) (\chi + t\chi + \tilde \chi + \sqrt{-1} \partial\bar\partial \varphi_2) + r \left(\chi + \dfrac{t}{2} \chi + \dfrac{1}{2} \tilde \chi + \sqrt{-1} \partial\bar\partial v_t\right) \right)^m \wedge \omega^{n - m} ,
\end{aligned}
\end{equation}
for any $ r \in [0,1]$. 
We can decompose
\begin{equation*}
\begin{aligned}
	&	\chi + t \chi + \tilde \chi + \sqrt{-1} \partial\bar\partial\varphi_1 \\
	=& (1 - r) \left(\chi + t \chi + \tilde \chi + \sqrt{-1} \partial\bar\partial\varphi_2\right)  + r \left(\chi + \frac{t}{2} \chi + \frac{1}{2} \tilde \chi + \sqrt{-1} \partial\bar\partial v_t\right) + \hat\chi 
	,
\end{aligned}
\end{equation*}
where 
\begin{equation*}
\hat \chi :=  \dfrac{r}{2} \left(t \chi + \tilde \chi\right) - r \sqrt{-1}  \partial\bar\partial v_t - (1 - r) \sqrt{-1} \partial\bar\partial\varphi_2 + \sqrt{-1} \partial\bar\partial\varphi_1 .
\end{equation*}
At any point with $\hat\chi \geq 0$, we derive from \eqref{inequality-5-1-1} that
\begin{equation*}
\label{inequality-5-1-2}
\begin{aligned}
	b_t f_1 \omega^n
	=& (\chi + t\chi + \tilde \chi + \sqrt{-1} \partial\bar\partial\varphi_1)^n - c (\chi + t\chi + \tilde \chi +  \sqrt{-1} \partial\bar\partial \varphi_1)^m \wedge \omega^{n - m} \\
	\geq&  \left( (1 - r) (\chi + t\chi + \tilde \chi + \sqrt{-1} \partial\bar\partial\varphi_2) + r \bigg(\chi + \frac{t}{2} \chi + \frac{1}{2} \tilde \chi +  \sqrt{-1} \partial\bar\partial v_t\bigg)  \right)^n \\
			& \quad -  c   \Bigg( (1 - r) (\chi + t\chi + \tilde \chi + \sqrt{-1} \partial\bar\partial\varphi_2) \\
			&\qquad \qquad + r \bigg(\chi + \frac{t}{2} \chi + \frac{1}{2} \tilde \chi +  \sqrt{-1} \partial\bar\partial v_t \bigg)  \Bigg)^m \wedge \omega^{n - m}  +   \hat\chi^n \\
	\geq&\; \hat\chi^n 
	,
\end{aligned}
\end{equation*}
that is,
\begin{equation*}
	b_t f_1 \omega^n
	\geq \frac{r^n}{2^n} \left(t \chi + \tilde \chi - 2 \sqrt{-1}  \partial\bar\partial v_t - \frac{1 - r}{r} \sqrt{-1} \partial\bar\partial \varphi_2 + \frac{1}{r} \sqrt{-1}  \partial\bar\partial \varphi_1\right)^n
	.
\end{equation*}

We solve the auxiliary Monge-Amp\`ere equation
\begin{equation*}
	\left(t\chi + \tilde \chi + \sqrt{-1}\partial\bar\partial \psi_k\right)^n = \frac{\tau_k \left(2 v_t + \frac{1 - r}{r} \varphi_2 - \frac{1}{r} \varphi_1 - s \right)}{A_{s,k}} \frac{2^n b_t f_1 }{r^n} \omega^n , \quad \sup_M \psi_k = 0,
\end{equation*}
where
\begin{equation*}
	V_t := \int_M \left(t\chi + \tilde \chi\right)^n 
\end{equation*}
and
\begin{equation*}
	A_{s,k} :=  \frac{1}{V_t} \int_M \tau_k \left(2 v_t + \frac{1 - r}{r} \varphi_2 - \frac{1}{r} \varphi_1 - s \right)  \frac{2^n b_t f_1 }{r^n} \omega^n .
\end{equation*}
We may assume that $s \geq \Vert  \varphi_1 \Vert_{L^\infty}$, which is uniformly bounded as shown in Section~\ref{L-estimate}. 
Applying the maximum principle to the following continuous function
\begin{equation*}
	\Phi := 
	2v_t + \frac{1 - r}{r} \varphi_2 - \frac{1}{r} \varphi_1 - s - \left( - \frac{n + 1}{n} A^{\frac{1}{n}}_{s,k} \psi_k + A^{\frac{n + 1}{n}}_{s,k} \right)^{\frac{n}{n + 1}}.
\end{equation*}
Let 
\begin{equation*}
M_s := \left\{ - 2 v_t - \frac{1 - r}{r} \varphi_2  + \frac{1}{r} \varphi_1 \leq - s \right\} .
\end{equation*}
Because of compactness of $M$, function $\Phi$ has to achieve its maximum at a certain point on $M$, say $z_0$. 
If $z_0 \in M\setminus \mathring{M}_s$, then obviously
\begin{equation*}
	\Phi (z_0)
	\leq 2 v_t + \frac{1 - r}{r} \varphi_2 - \frac{1}{r} \varphi_1 - s \leq 0.
\end{equation*} 
Otherwise,  we have that at point $z_0$,
\begin{equation}
\label{inequality-6-15}
\begin{aligned}
\frac{2}{r}\hat\chi
	\geq& 
	 \left( - \frac{n + 1}{n} A^{\frac{1}{n}}_{s,k} \psi_k + A^{\frac{n + 1}{n}}_{s,k}\right)^{- \frac{1}{n + 1}}   A^{\frac{1}{n}}_{s,k}  \left( t\chi + \tilde \chi +  \sqrt{-1} \partial\bar\partial (\psi_k)  \right)\\
	 &\; + \left(1 - \left( - \frac{n + 1}{n} A^{\frac{1}{n}}_{s,k} \psi_k + A^{\frac{n + 1}{n}}_{s,k}\right)^{- \frac{1}{n + 1}}   A^{\frac{1}{n}}_{s,k}  \right) (t\chi + \tilde \chi) \\
	 \geq& 
	 	 \left( - \frac{n + 1}{n} A^{\frac{1}{n}}_{s,k} \psi_k + A^{\frac{n + 1}{n}}_{s,k}\right)^{- \frac{1}{n + 1}}   A^{\frac{1}{n}}_{s,k}  \left( t\chi + \tilde \chi +  \sqrt{-1} \partial\bar\partial  \psi_k \right) .
\end{aligned}
\end{equation}
Considering the volume form of \eqref{inequality-6-15},
\begin{equation*}
\begin{aligned}
	  \frac{2^n}{r^n} b_t f_1 \omega^n 
	\geq& 
	\left( - \frac{n + 1}{n} A^{\frac{1}{n}}_{s,k} \psi_k + A^{\frac{n + 1}{n}}_{s,k}\right)^{- \frac{n}{n + 1}}     \left( 2 v_t + \frac{1 - r}{r} \varphi_2 - \frac{1}{r} \varphi_1 - s \right) \frac{2^n}{r^n} b_t f_1 \omega^n ,
\end{aligned}
\end{equation*}
that is, 
\begin{equation*}
\left( - \frac{n + 1}{n} A^{\frac{1}{n}}_{s,k} \psi_k + A^{\frac{n + 1}{n}}_{s,k}\right)^{ \frac{n}{n + 1}}    \geq 2 v_t + \frac{1 - r}{r} \varphi_2 - \frac{1}{r} \varphi_1 - s .
\end{equation*}

Since $\chi + \tilde \chi + \sqrt{-1} \partial\bar\partial\psi_k >0$, 
there is a constant $\beta > 0$ such that
\begin{equation}
\label{inequality-6-18}
\begin{aligned}
	&\quad \int_{M_s} \exp \left( \frac{\beta}{A^{\frac{1}{n}}_{s,k}}   \left( 2 v_t + \frac{1 - r}{r} \varphi_2 - \frac{1}{r} \varphi_1 \right)^{\frac{n + 1}{n}}    \right) \omega^n \\
	&\leq  \int_{M_s} \exp \left( - \frac{n + 1}{n} \beta \psi_k + \beta  A_{s,k} \right) \omega^n \\
	&\leq C e^{\beta A_{s,k}} .
\end{aligned}
\end{equation}
Letting $k \to + \infty$, we obtain
\begin{equation}
\label{inequality-6-19}
		\int_{M_s} \exp \left( \frac{\beta}{A^{\frac{1}{n}}_{s}}   \left( 2 v_t + \frac{1 - r}{r} \varphi_2 - \frac{1}{r} \varphi_1 \right)^{\frac{n + 1}{n}}    \right) \omega^n 
		\leq C e^{\beta A_{s}} ,
\end{equation}
where
\begin{equation*}
	A_s := \lim_{k \to +\infty} A_{s,k} = \frac{1}{V_t} \int_{M_s} \left(2 v_t + \frac{1 - r}{r} \varphi_2 - \frac{1}{r} \varphi_1 - s \right) \frac{2^n b_t f_1}{r^n} \omega^n .
\end{equation*}

Now we try to find the estimates of $A_s$. Since $s \geq \Vert \varphi_1 \Vert_{L^\infty}$, 
\begin{equation*}
	\frac{1 - r}{r} (\varphi_2 - \varphi_1 ) \geq - 2 v_t + \varphi_1 + s \geq \varphi_1 + s \geq 0 ,   \qquad\text{ on } M_s .
\end{equation*}
By direction computation,
\begin{equation}
\label{inequality-6-20}
\begin{aligned}
	A_s 
	 \leq&\,  \frac{1}{V_t} \int_{M_s} \left(\frac{1 - r}{r} (\varphi_2 - \varphi_1) - \varphi_1    -  s\right) \frac{2^n b_t f_1}{r^n} \omega^n \\
	 \leq&\, \frac{1}{V_t } \int_{M_s} \frac{1 - r}{r} (\varphi_2 - \varphi_1)^+ \frac{2^n b_t f_1}{r^n} \omega^n \\
	 \leq&\, \frac{2^n b_t}{V_t r^{n + 1}} \int_{M_s} (\varphi_2 - \varphi_1)^+ f_1 \omega^n 
	 .
\end{aligned}
\end{equation}
Applying H\"older inequality to \eqref{inequality-6-20} with respect to measure $\omega^n$,
\begin{equation*}
		A_s 
		\leq 
		\frac{2^n b_t}{V_t r^{n + 1}} \Vert (\varphi_2 - \varphi_1)^+ \Vert_{L^{q*}} \Vert f_1 \Vert_{L^q}  
		\leq 
		\frac{2^n b_1}{V_0 r^{n + 1}} \Vert (\varphi_2 - \varphi_1)^+ \Vert_{L^{q*}} \Vert f_1 \Vert_{L^q}  
		.
\end{equation*}
Since $\Vert (\varphi_2 - \varphi_1)^+\Vert_{L^{q*}} < K$, we can choose
\begin{equation*}
	r := \frac{1}{2} \left(    \frac{\Vert (\varphi_2 - \varphi_1)^+\Vert_{L^{q*}} }{K}  \right)^{\frac{1}{n + 1}} < \frac{1}{2} ,
\end{equation*}
and hence we denote the upper bound for $A_s$ by 
\begin{equation*}
\displaystyle
	E 
	:= \dfrac{2^{2n + 1} b_1}{V_0  } K   \Vert f_1 \Vert_{L^q}  
	.
\end{equation*}

Now we shall apply a De Giorigi iteration argument. For any $p > n$ and $q > 1$,
\begin{equation}
\label{inequality-6-22}
\begin{aligned}
	A_s
	\leq&\, 
	\frac{2^n b_t}{V_t r^n}  \left(  \int_{M_s } f_1 \omega^n\right)^{\frac{(n + 1)p - n}{(n + 1) p}} \\
	&\quad \cdot   \left(   \int_{M_s}  \left(2 v_t + \frac{1 - r}{r} \varphi_2 - \frac{1}{r} \varphi_1 - s\right)^{\frac{(n + 1) p}{n}} f_1 \omega^n\right)^{\frac{n}{(n + 1) p}} \\
	\leq&\,
			 \frac{2^n b_t}{V_t r^n}   \frac{A^{\frac{1}{n + 1}}_s}{\beta^{\frac{n}{n + 1}}} \phi^{\frac{(n + 1)p - n}{(n + 1) p}}  (s)  \Vert f_1 \Vert^{\frac{n}{(n+ 1) p}} _{L^q} \\
			  &\quad\cdot
			  			  \left(  \Bigg(\int_{M_s} \bigg( \frac{\beta}{A^{\frac{1}{n}}_s} \left(2 v_t + \frac{1 - r}{r} \varphi_2 - \frac{1}{r} \varphi_1 - s \right)^{\frac{n + 1 }{n}} \bigg)^{\frac{pq}{q - 1}} \omega^n\Bigg)^{1 - \frac{1}{q}}   \right)^{\frac{n}{(n+ 1) p}} 
			  \\
	\leq&\,
				 C(p,q) \frac{b_t}{V_t r^n}  \frac{A^{\frac{1}{n + 1}}_s}{\beta^{\frac{n}{n + 1}}}
				  e^{\frac{n \beta}{(n + 1)p} \left(1 - \frac{1}{q}\right) A_s}  \Vert f_1 \Vert^{\frac{n}{(n+ 1) p}} _{L^q} 
				  \phi ^{\frac{(n + 1)p - n}{(n + 1) p}} (s) ,
\end{aligned}
\end{equation}
where
\begin{equation*}
	\phi (s) := \int_{M_s} f_1 \omega^n .
\end{equation*}
In the above computation, we apply H\"older inequality with respect to measure $f_1 \omega^n$ in the first inequality, and with respect to measure $\omega^n$ in the second inequality.  The last inequality is due to \eqref{inequality-6-19}.
As in Section~\ref{L-estimate}, we can pick $p = n + 1$ in this proof. 
Rearranging \eqref{inequality-6-22},
\begin{equation*}
\begin{aligned}
	A_s
	\leq&\,
				 C(p,q) \frac{b^{\frac{n + 1}{n}}_t}{V^{\frac{n + 1}{n}}_t r^{n + 1}}  \frac{1}{\beta}
				  e^{\frac{ \beta}{p} \left(1 - \frac{1}{q}\right) A_s}  \Vert f_1 \Vert^{\frac{1}{ p}} _{L^q} 
				  \phi^{\frac{(n + 1)p - n}{np}}  (s) \\
	\leq&\, 
					 C(p,q) \frac{b^{\frac{n + 1}{n}}_t}{V^{\frac{n + 1}{n}}_t r^{n + 1}}  \frac{1}{\beta}
					  e^{\frac{ \beta}{p} \left(1 - \frac{1}{q}\right) E}  \Vert f_1 \Vert^{\frac{1}{ p}} _{L^q} 
					  \phi^{\frac{(n + 1)p - n}{np}} (s) \\
	\leq&\,
		 C(p,q, \beta, \Vert f_1\Vert_{L^q}, E) \frac{b^{\frac{n + 1}{n}}_t}{V^{\frac{n + 1}{n}}_t r^{n + 1}}  
						  \phi^{\frac{(n + 1)p - n}{np}}  (s) .
\end{aligned}
\end{equation*}
For any $s' > 0$ and $s \geq \Vert \varphi_1\Vert_{L^\infty}$,
\begin{equation}
\label{inequality-6-26}
\begin{aligned}
	s' \phi (s + s')
	&\leq \int_{M_{s+ s'}}  \left(2 v_t + \frac{1 - r}{r} \varphi_2 - \frac{1}{r} \varphi_1 - s\right) f_1 \omega^n \\
	&\leq  \int_{M_{s}}  \left(2 v_t + \frac{1 - r}{r} \varphi_2 - \frac{1}{r} \varphi_1 - s\right) f_1 \omega^n \\
	&\leq C(p,q, \beta, \Vert f_1\Vert_{L^q}, E)  \frac{b^{\frac{1}{n}}_1}{V^{\frac{1}{n}}_0 r}  
								  \phi^{1 + \frac{1}{n} - \frac{1}{p}}  (s) .
\end{aligned}
\end{equation}
Combining Lemma~\ref{de-giorgi-iteration} and Inequality~\eqref{inequality-6-26}, we can conclude that
\begin{equation*}
	\int_{M_s} f_1 \omega^n = 0 ,
\end{equation*}
when
\begin{equation*}
s \geq s_\infty := \Vert \varphi_1\Vert_{L^\infty} + C(p,q, \beta, \Vert f_1\Vert_{L^q}, E)  2^{\frac{1 + \frac{1}{n} - \frac{1}{p}}{\frac{1}{n} - \frac{1}{p}}} \left(\left(\int_M  \omega^n\right)^{1 - \frac{1}{q}} \Vert f_1\Vert_{L^q} \right)^{\frac{1}{n} - \frac{1}{p}} .
\end{equation*}
Consequently
\begin{equation*}
\begin{aligned}
 \varphi_2 - \varphi_1 
 &\leq 
2 \Bigg(2 \Vert v_t \Vert_{L^\infty} + \Vert \varphi_1\Vert_{L^\infty} \\
&\qquad  \qquad+ C(p,q, \beta, \Vert f_1\Vert_{L^q}, E)  2^{\frac{1 + \frac{1}{n} - \frac{1}{p}}{\frac{1}{n} - \frac{1}{p}}} \left(\left(\int_M  \omega^n\right)^{1 - \frac{1}{q}} \Vert f_1\Vert_{L^q} \right)^{\frac{1}{n} - \frac{1}{p}} \Bigg) r
.
\end{aligned}
\end{equation*}

%
%
%
%
%
%
%
%
%

\end{proof}

\medskip
\section{The solution in the boundary case}
\label{solution}

In this section, we shall study the solution to Equation~\eqref{equation-2} in the boundary case. 
The boundary case is similar to the singular complex Monge-Amp\`ere equation, and the solution is achieved in pluripotential sense. The traditional viscosity method fails if we do not impose any extra geometric condition, since we can only show uniform smoothness on the ample locus of $\tilde\chi$~\cite{EyssidieuxGuedjZeriahi2011}\cite{EyssidieuxGuedjZeriahi2017}.

To construct a pluripotential solution, we need the following precompactness result~\cite{Sun2022}. For completeness, we include a proof here.
\begin{lemma}
	\label{lemma-7-1}
	For $0 < t \leq 1$, the set of $L^\infty$ bounded admissible solutions to approximation equation~\eqref{approximation-equation-1-5} is precompact in  $L^{q^*}$ norm for $1 \leq q^* < + \infty$.
	
\end{lemma}  
\begin{proof}

Noticing that $2 \chi + \tilde \chi + \sqrt{-1}\partial\bar\partial \varphi$,
we have
\begin{equation}
\begin{aligned}
	\Vert \varphi \Vert_{L^\infty} \int_M (2 \chi + \tilde \chi ) \wedge \omega^{n - 1}
	&\geq \int_M - \varphi (2 \chi + \tilde \chi + \sqrt{-1} \partial\bar\partial \varphi) \wedge \omega^{n - 1} \\
	&\geq \int_M - \varphi \sqrt{-1} \partial\bar\partial \varphi \wedge \omega^{n - 1} \\
	&= \Vert \nabla \varphi \Vert^2_{L^2 (\omega^n)}
	.
\end{aligned}
\end{equation}
By compact embedding for Sobolev spaces, the set of $L^\infty$ bounded admissible solutions is precompact in $L^1$ norm. 
For a $L^1$-convergent sequence $\{\varphi_i\}$, there is a subsequence $\varphi_{i_j}$ which converges almost everywhere. Then for any $1 \leq q^* < +\infty$,
\begin{equation}
\begin{aligned}
	\int_M \left| \varphi_{i_j} - \varphi_{i_k}\right|^{q^*}\omega^n
	&\leq 
	\Vert \varphi_{i_j} - \varphi_{i_k}\Vert^{q^* - 1}_{L^\infty} \int_M \left| \varphi_{i_j} - \varphi_{i_k}\right| \omega^n \\
	&\leq 
	C \int_M \left| \varphi_{i_j} - \varphi_{i_k}\right| \omega^n  \\
	&\to 0 \qquad (j,k \to \infty) ,
\end{aligned}
\end{equation}
that is, $\{\varphi_{i_j}\}$ is Cauchy in $L^{q^*}$. 
Therefore, $L^\infty$-bounded set of admissible solutions is precompact in $L^{q^*}$.

\end{proof}
%
%
%

%
%
%
%
%
%
%
%
%
%
%
%
%
%
%
%
%
%
%
Now we shall construct a weak solution in pluripotential sense to Equation~\eqref{equation-2} in the boundary case.
\begin{theorem}
\label{theorem-7-2}

For $q > 1$ and $f \in L^q$, there is a bounded solution $\varphi$ in pluripotential sense to Equation~\eqref{equation-2} in the boundary case~\eqref{condition-1}. In particular,
\begin{equation*}
 \left(\chi  + \tilde \chi + \sqrt{-1} \partial\bar\partial \varphi \right)^n \geq c^{\frac{n}{n - m}} \omega^n .
\end{equation*}
\end{theorem}
\begin{proof}

For $0 < t \leq 1$, there exists a set of solutions to approximation equation~\eqref{approximation-equation-1-5} with condition~\eqref{condition-1-6}. 
According to Lemma~\ref{lemma-7-1}, this solution set is uniformly bounded, and hence precompact in $L^{q^*}$ norm where $q^* = \frac{q}{q - 1}$. Then we can choose a decreasing sequence $t_i \to 0+$ such that $\varphi_{t_i}$ is convergent in $L^{q^*}$ norm. For all $j \in \mathbb{N}$, there is $i_j \in \mathbb{N}$ such that for any $i' > i_j$,
\begin{equation*}
	\Vert \varphi_{t_i'} - \varphi_{t_{i_j}}\Vert_{L^{q^*}} < \frac{1}{2^{(n+1) j}} .
\end{equation*}

For $0 < t' < t \leq 1$, we have
\begin{equation*}
\chi + t\chi + \tilde \chi + \sqrt{-1} \partial\bar\partial \varphi_{t'} > 0 ,
\end{equation*}
and
\begin{equation*}
	(\chi + t\chi + \tilde \chi + \sqrt{-1} \partial\bar\partial \varphi_{t'})^n = c (\chi + t\chi + \sqrt{-1} \partial\bar\partial \varphi_{t'})^m \wedge \omega^{n - m} + b_{t} g \omega^n
\end{equation*}
for some function $b_t g > b_{t'} f > 0$. 
According to Theorem~\ref{theorem-6-1},
\begin{equation}
	\sup_M (\varphi_{t_{i_{j + 1}}} - \varphi_{t_{i_j}}) < C(\chi , \tilde \chi , \omega , 1,  q, \Vert f\Vert_{L^q}) \frac{1}{2^j} .
\end{equation}
and hence $\left\{\varphi_{t_{i_j}} + \frac{C}{2^{j - 1}}\right\}$ is pointwisely decreasing and bounded. So sequence 
$\left\{\varphi_{t_{i_j}} + \frac{C}{2^{j - 1}}\right\}$ decreases to some bounded function $\varphi$.
Letting $j \to \infty$ in 
\begin{equation*}
	\left(\chi + t_{i_j} \chi + \tilde \chi + \sqrt{-1} \partial\bar\partial \varphi_{t_{i_j}}\right)^n
	=
	c 	\left(\chi + t_{i_j} \chi + \tilde \chi + \sqrt{-1} \partial\bar\partial \varphi_{t_{i_j}}\right)^m \wedge \omega^{n - m} + b_{t_{i_j}} f \omega^n ,
\end{equation*}
we see that $\varphi$ is a solution in pluripotential sense to Equation~\eqref{equation-2} after normalization.

Newton-Maclaurin inequality tells us that
\begin{equation*}
\begin{aligned}
	&\qquad \frac{\left(\chi + t_{i_j} \chi + \tilde \chi + \sqrt{-1} \partial\bar\partial \varphi_{t_{i_j}}\right)^n}{\omega^n} \\
	&\geq 
	\left(\frac{\left(\chi + t_{i_j} \chi + \tilde \chi + \sqrt{-1} \partial\bar\partial \varphi_{t_{i_j}}\right)^n}{\left(\chi + t_{i_j} \chi + \tilde \chi + \sqrt{-1} \partial\bar\partial \varphi_{t_{i_j}}\right)^m \wedge \omega^{n - m}}\right)^{\frac{n}{n - m}} \\
	&\geq 
	c^{\frac{n}{n - m}}
	,
\end{aligned}
\end{equation*}
and hence we obtain
\begin{equation*}
 \left(\chi  + \tilde \chi + \sqrt{-1} \partial\bar\partial \varphi \right)^n \geq c^{\frac{n}{n - m}} \omega^n .
\end{equation*}

\end{proof}

The function set $\{\varphi_{t_i}\}$ is bounded in $W^{1,2}$, and hence there is a subsequence weakly convergent to a function $\tilde \varphi \in W^{1,2}$ by Alaoglu's Theorem. 
Since $W^{1,2}$ is compacted embedded in $L^2$, there is a convergent subsequence in $L^2$ norm.  
By passing to a subsequence again, we can assume that $\{\varphi_{t_i}\}$ is convergent to $\tilde\varphi$ almost everywhere, and hence $\tilde \varphi = \varphi$ as defined in Theorem~\ref{theorem-7-2}. 
Moreover, by a similar argument to \cite{Cegrell2007}, we can show that $\varphi_{t_i}$ is convergent in $W^{1,2}$ to $\varphi$.

Moreover, as stated in Section~\ref{smoothness}, $\{\varphi_{t_i}\}$ has a uniform  estimates in $Amp(\tilde \chi)$. Passing to a subsequence if necessary, $\{\varphi_{t_i}\}$ locally uniformly converges   to a smooth function in $ Amp(\tilde\chi)$. It is easy to see that the function has to be $\varphi$ defined in Theorem~\ref{theorem-7-2}. 

\medskip
\section{Uniqueness}
\label{uniqueness}

In previous section, we construct admissible solutions to Equation~\eqref{equation-2} from approximations equation~\eqref{approximation-equation-1-5} with variant functions $f$. In this section, we shall show that these admissible solutions are  same in some sense.

We shall adapt Calabi's trick. Suppose that there are two solutions $\varphi_1$ and $\varphi_2$ as constructed in Section~\ref{solution}, respectively from
\begin{equation}
\label{construction-8-1}
	(\chi + t_i \chi + \tilde \chi + \sqrt{-1} \partial\bar\partial \varphi_{t_i})^n = c (\chi + t_i \chi + \tilde \chi + \sqrt{-1} \partial\bar\partial \varphi_{t_i})^m \wedge \omega^{n - m} + b_{t_i} f_1 \omega^n
\end{equation}
and
\begin{equation}
\label{construction-8-2}
	(\chi + t'_j \chi + \tilde \chi + \sqrt{-1} \partial\bar\partial \varphi_{t'_j})^n = c (\chi + t'_j \chi + \tilde \chi + \sqrt{-1} \partial\bar\partial \varphi_{t'_j})^m \wedge \omega^{n - m} + b_{t'_j} f_2 \omega^n .
\end{equation}
Combining \eqref{construction-8-1} and \eqref{construction-8-2},
\begin{equation*}
\begin{aligned}
	&\quad \int_M \left(\varphi_{t'_j} - \varphi_{t_i}\right) \left(b_{t_i} f_1  - b_{t'_j} f_2 \right) \omega^n \\
	&= \int^1_0 \Bigg(\int_M \left(\varphi_{t'_j} - \varphi_{t_i}\right) \left( (t_i - t'_j) \chi +  \sqrt{-1} \partial\bar\partial \left(\varphi_{t_i} - \varphi_{t'_j}\right) \right) \\
	&\qquad \quad \quad \wedge \Bigg(  n \left(\chi + \left(\theta t_i  + (1 - \theta) t'_j \right) \chi + \tilde \chi + \sqrt{-1}  \partial\bar\partial\left(\theta  \varphi_{t_i} + (1 - \theta)  \varphi_{t'_j}\right)\right)^{n - 1}     \\
	\\
	&\qquad \qquad  \quad\quad  - m  c\left(\chi + \left(\theta t_i   + (1 - \theta) t'_j \right)\chi + \tilde \chi + \sqrt{-1}  \partial\bar\partial\left(\theta  \varphi_{t_i} + (1 - \theta)  \varphi_{t'_j}\right)\right)^{m - 1}  \\
	&\qquad \qquad \qquad \qquad \qquad  \qquad \qquad\wedge \omega^{n - m}  \Bigg)\Bigg) d\theta 
	.
\end{aligned}
\end{equation*}
Applying Stokes' Lemma and letting $i, j \to \infty$,
\begin{equation*}
\begin{aligned}
	0
	&=   \int_M \sqrt{-1} \partial \left( \varphi_1 - \varphi_2 \right) \wedge     \bar\partial \left(\varphi_1 - \varphi_2 \right)   \\
	&\qquad \qquad  \wedge \Bigg(\int^1_0\Big(  n \left(\chi +   \tilde \chi + \sqrt{-1}  \partial\bar\partial\left(\theta  \varphi_1 + (1 - \theta)  \varphi_2 \right)\right)^{n - 1}     \\
	\\
	&\qquad \qquad \qquad \qquad\quad  - m c\left(\chi +   \tilde \chi + \sqrt{-1}  \partial\bar\partial\left(\theta  \varphi_1 + (1 - \theta)  \varphi_2 \right)\right)^{m - 1}  \wedge \omega^{n - m}  \Big)d\theta \Bigg) 
	.
\end{aligned}
\end{equation*}
Since the complement of $Amp (\tilde \chi)$ is codimensional at least $2$,  $Amp (\tilde \chi)$ is connected and the complement has measure zero. 
In the ample locus, $\varphi_1$ and $\varphi_2$ are smooth, and hence by concavity
\begin{equation*}
\begin{aligned}
	&\quad 
	\dfrac{\left(\chi + \tilde \chi + \sqrt{-1} \partial\bar\partial \left(\theta \varphi_1 + (1 - \theta) \varphi_2\right) \right)^n}{\left(\chi + \tilde \chi + \sqrt{-1} \partial\bar\partial \left(\theta \varphi_1 + (1 - \theta) \varphi_2\right) \right)^m \wedge \omega^{n - m}} 
	&\geq 
	c
	.
\end{aligned}
\end{equation*}
By monotonicity of $\frac{S_n}{S_m}$,
\begin{equation*}
\begin{aligned}
	0 &<
	n \left(\chi + \tilde \chi + \sqrt{-1} \partial\bar\partial \left(\theta \varphi_1 + (1 - \theta) \varphi_2\right) \right)^{n - 1 } \\
	&\qquad -
	m c \left(\chi + \tilde \chi + \sqrt{-1} \partial\bar\partial \left(\theta \varphi_1 + (1 - \theta) \varphi_2\right) \right)^{m - 1} \wedge \omega^{n - m} .
\end{aligned}
\end{equation*}
So we can see that $\varphi_1 - \varphi_2$ is constant in $Amp (\tilde \chi)$.

\medskip
\section{Fake boundary case}
\label{fake}

In \cite{Sun2016}, the author solved
\begin{equation}
\label{equation-9-1}
	\left(\chi + \sqrt{-1} \partial\bar\partial \varphi\right)^n = g \left(\chi + \sqrt{-1} \partial\bar\partial \varphi\right)^m \wedge \omega^{n - m},
\end{equation}
where $g $ is a smooth function and 
\begin{equation*}
	g \geq c:= \frac{\int_M \chi^n}{\int_M \chi^m \wedge \omega^{n - m}} .
\end{equation*}
The boundary case of the cone condition seems to be
\begin{equation}
\label{fake-boundary-case}
	n \chi^{n - 1} - m g \chi^{m - 1} \wedge \omega^{n - m} \geq 0 . 
\end{equation}
However, we did not take into account the general case $g \geq c$ in Theorem~\ref{main-theorem},
because conditin~\eqref{fake-boundary-case} is actually a fake boundary case if $g \not \equiv c$. In this section, we shall show that Equation~\eqref{equation-9-1} has a classical (smooth) solution.

When $g \not\equiv c$ on $M$, the solvability of Equation~\eqref{equation-9-1} means that there exist a smooth function $\varphi$ and a constant $b$ solving
\begin{equation}
\label{equation-9-3}
	\left(\chi + \sqrt{-1} \partial\bar\partial \varphi\right)^n = e^b g \left(\chi + \sqrt{-1} \partial\bar\partial \varphi\right)^m \wedge \omega^{n - m}.
\end{equation}
Since $g\geq c$, it is easy to see that $b \leq 0$ by integrating Equation~\eqref{equation-9-3}. 
However we do not have any way to determine the exact value of $b$ a priori.
Since $g$ is continuous, there is a maximal value $\Lambda > c$ for $g$.

If $\min_M g = \lambda > c$, then we instead consider
\begin{equation*}
\label{equation-9-4-1}
	\left(\chi + \sqrt{-1} \partial\bar\partial \varphi\right)^n = \frac{c}{\lambda} g \left(\chi + \sqrt{-1} \partial\bar\partial \varphi\right)^m \wedge \omega^{n - m} .
\end{equation*}
We derive from \eqref{fake-boundary-case} that
\begin{equation*}
		n \chi^{n - 1} - m \frac{c}{\lambda} g \chi^{m - 1} \wedge \omega^{n - m} > 
		0 . 
\end{equation*}
According to \cite{Sun2016}\cite{Sun2017}, there is a smooth admissible solution $\varphi$ to Equation~\eqref{equation-9-1} under the sense of \eqref{equation-9-3}.
Therefore, we only need to consider the case of $\min_M g = c$ in the rest of this section.

\medskip
\subsection{An upper bound for constant $b$}
\label{upper-bound-b}

Assume that we have a $C^2$ admissible solution to Equation~\eqref{equation-9-3}. By Newton-Maclaurin inequality, 
\begin{equation*}
\begin{aligned}
	\left(\frac{\left(\chi + \sqrt{-1} \partial\bar\partial \varphi\right)^m \wedge \omega^{n - m}}{\omega^n}\right)^{\frac{1}{m}}
	&\geq 
	\left(\frac{\left(\chi + \sqrt{-1} \partial\bar\partial \varphi\right)^n}{ \left(\chi + \sqrt{-1} \partial\bar\partial \varphi\right)^m \wedge \omega^{n - m}}\right)^{\frac{1}{n - m}} 
	&\geq e^{\frac{b}{n - m}} c^{\frac{1}{n - m}}
	.
\end{aligned}
\end{equation*}
That is,
\begin{equation}
 \left(\chi + \sqrt{-1} \partial\bar\partial \varphi\right)^m \wedge \omega^{n - m} 
 \geq 
 e^{\frac{m b}{n - m}} c^{\frac{m}{n - m}} \omega^n .
\end{equation}
Integrating Equation~\eqref{equation-9-3}, 
\begin{equation}
\label{inequality-9-6}
\begin{aligned}
	\int_M \chi^n 
	&= e^b c \int_M \chi^m \wedge \omega^{n - m} + e^b \int_M (g - c) (\chi + \sqrt{-1} \partial\bar\partial \varphi)^m \wedge \omega^{n - m} \\
	&\geq e^b c \int_M \chi^m \wedge \omega^{n - m} + \frac{\Lambda - c}{2} e^b \int_{\left\{g \geq \frac{\Lambda + c}{2}\right\}} (\chi + \sqrt{-1} \partial\bar\partial \varphi)^m \wedge \omega^{n - m}  \\
	&\geq 
	e^b c \int_M \chi^m \wedge \omega^{n - m} + \frac{\Lambda - c}{2} e^{\frac{n b}{n - m}} c^{\frac{m}{n - m}} \int_{\left\{g \geq \frac{\Lambda + c}{2}\right\}} \omega^n 
	.
\end{aligned}
\end{equation}
There is a constant $\theta_0 > 0$ such that
\begin{equation}
\label{inequality-9-7}
	  \frac{\Lambda - c}{2}  c^{\frac{m}{n - m}} \int_{\left\{g \geq \frac{\Lambda + c}{2}\right\}} \omega^n 
	  \geq
	  \theta_0 c \int_M \chi^m \wedge \omega^{n - m} .
\end{equation}
Substituting \eqref{inequality-9-7} into \eqref{inequality-9-6}, 
\begin{equation}
	1 \geq e^b + \theta_0 e^{\frac{n b}{n - m}} .
\end{equation}
There is a unique solution $b' < 0$ to $1 = e^x + \theta_0 e^{\frac{n x}{n - m}}$, and consequently $b \leq b'$.

\medskip
\subsection{Solving complex Monge-Amp\`ere type equations}

Now we begin to solve the complex Monge-Amp\`ere type equations through method of continuity, which was carried out in \cite{Sun2016}. Following the work in \cite{Sun2016}, we need to apply this method of continuity twice.

First, we define a smooth function $g_1$ by
\begin{equation}
\label{equality-9-9}
	\chi^n = g_1 \chi^m \wedge \omega^{n - m} .
\end{equation}
Then
\begin{equation}
\label{cone-9-10}
	n \chi^{n - 1} > m \left( \max\left\{e^{b'} g,g_1 \right\} + 2\delta_1 \right) \chi^{n - 1} \wedge \omega^{n - m} .
\end{equation}
for some $\delta_1 > 0$ sufficiently small. 
We can find a smooth function $g_2$ satisfying
\begin{equation*}
	\max\left\{e^{b'} g,g_1 \right\} < g_2 < \max\left\{e^{b'} g,g_1 \right\} + \delta_1 .
\end{equation*}
Therefore, we have from \eqref{equality-9-9} and \eqref{cone-9-10} that
\begin{equation}
	\chi^n < g_2 \chi^m \wedge \omega^{n - m} ,
\end{equation}
and 
\begin{equation}
\label{cone-9-12}
	n \chi^{n - 1} > m \left(g_2 + \delta_1\right) \chi^{m - 1} \wedge \omega^{n - m} .
\end{equation}
According to Theorem 1 in \cite{Sun2016}, there exists a smooth function $\tilde \varphi$ and a unique constant $\tilde b < 0$ solving
\begin{equation*}
	\chi^n_{\tilde \varphi} 
	=
	e^{\tilde b} g_2 \chi^m_{\tilde \varphi} \wedge \omega^{n - m} .
\end{equation*}

Second, we adopt method of continuity again from $\chi_{\tilde \varphi}$ and consider the following path
\begin{equation}
\label{path-9-13}
	\chi^n_{\varphi_t} 
	= 
	e^{b_t + t b'} g^t g^{1 - t}_2 \chi^m_{\varphi_t} \wedge \omega^{n - m} ,\qquad \text{for } t\in [0,1] ,
\end{equation}
where $b_0 = \tilde b$ found out in the first stage. 
Thus
\begin{equation}
\label{inequality-9-14}
	\chi^n_{\varphi_t} 
	> 
	e^{b_t + b'}  g \chi^m_{\varphi_t} \wedge \omega^{n - m} 
	.
\end{equation}
Repeating the argument in subsection~\ref{upper-bound-b}, we have
\begin{equation}
 1 > e^{b_t +   b'} + \theta_0 e^{\frac{n (b_t +  b')}{n - m} },
\end{equation} 
and hence $b_t < 0$.
As a direct result,
\begin{equation}
e^{b_t + t b'} g^t g^{1 - t}_2 < e^{t b'} g^t g^{1 - t}_2 < g_2 .
\end{equation}
Therefore, we obtain $t$-independent $C^\infty$ estimates of $\varphi_t$ as in \cite{Sun2016}, and consequently a solution pair to Equation~\eqref{equation-9-3} by method of continuity. 

\medskip
\noindent
{\bf Acknowledgements}\quad
The author wish to thank  Chengjian Yao and Ziyu Zhang for their helpful discussions and suggestions. The  author is supported by a start-up grant from ShanghaiTech University.

\end{document}